\documentclass[reqno,11pt]{amsart}
\usepackage{amssymb, amsmath,latexsym,amsfonts,amsbsy,amsthm,mathtools}

\usepackage{cite}
\usepackage[
pdfencoding=auto,
colorlinks = true,
citecolor = black,
linkcolor = black,
anchorcolor = black,
urlcolor = black,
filecolor = black,
pdfkeywords={}
]{hyperref}
\usepackage{todonotes}
\makeatletter
\define@key{todonotes}{Mingwen}[]{%
	\setkeys{todonotes}{author=\textbf{Mingwen},inline,color=red!20}}%
\define@key{todonotes}{Xiang}[]{%
	\setkeys{todonotes}{author=\textbf{Xiang},inline,color=green!20}}%
\define@key{todonotes}{Yadong}[]{%
	\setkeys{todonotes}{author=\textbf{Yadong},inline,color=cyan!20}}%
\makeatother

\usepackage{xcolor}

\allowdisplaybreaks[4]

\newcommand{\beq}{\begin{equation}}
\newcommand{\eeq}{\end{equation}}
\newcommand{\ben}{\begin{eqnarray}}
\newcommand{\een}{\end{eqnarray}}
\newcommand{\beno}{\begin{eqnarray*}}
\newcommand{\eeno}{\end{eqnarray*}}

\newtheorem{theorem}{Theorem}[section]

\newtheorem{lemma}[theorem]{Lemma}
\newtheorem{proposition}[theorem]{Proposition}

\newtheorem{Theorem}{Theorem}[section]

\newtheorem{Proposition}[Theorem]{Proposition}
\newtheorem{assumption}[Theorem]{Assumption}
\newtheorem{Lemma}[Theorem]{Lemma}

\newtheorem{Remark}[Theorem]{Remark}

\begin{document}
\title[Quasi-incompressible two-phase flow]
{Local-in-time existence of strong solutions to a quasi-incompressible Cahn--Hilliard--Navier--Stokes system}

%\author{}
%\address{}
%\email{}
%
\author{Mingwen Fei}
\address{School of Mathematics and Statistics, Anhui Normal University, Wuhu 241002, P. R. China}
\email{mwfei@ahnu.edu.cn}
\author{Xiang Fei}
\address{School of Mathematics and Statistics, Anhui Normal University, Wuhu 241002, P. R. China}
\email{feixiang@ahnu.edu.cn}
\author{Daozhi Han}
\address{Department of Mathematics, The State University of New York at Buffalo, Buffalo, NY 14260 USA}
\email{daozhiha@buffalo.edu}
\author{Yadong Liu}
\address{School of Mathematical Sciences, Ministry of Education Key Laboratory of NSLSCS, and Key Laboratory of Jiangsu Provincial Universities of FDMTA, Nanjing Normal University, Nanjing 210023, P. R. China
}
\email{ydliu@njnu.edu.cn}
\date{\today}


%%%%%%%%%%%%%%%%%%%%%%%%%%%%%%%%%%%%%%%%%%%%%%
%%%%%%%%%%%%%%%%%%%%%%%%%%%%%%%%%%%%%%%%%%
\begin{abstract}
We analyze a quasi-incompressible Cahn--Hilliard--Navier--Stokes system (qCHNS) for two-phase flows with unmatched densities. The order parameter is the volume fraction difference of the two fluids, while mass-averaged velocity is adopted. This leads to a quasi-incompressible model where the pressure also enters the equation of the  chemical potential.
We establish local existence and uniqueness of strong solutions by the Banach fixed point theorem and the maximal regularity theory.  
\end{abstract}

\numberwithin{equation}{section}

\indent

\maketitle
\section{Introduction}

In this article we consider the quasi-incompressible Cahn--Hilliard--Navier--Stokes system (qCHNS) for two-phase flows with mismatched densities \cite{GCLLJ,SYW,MGEK,MKID} in a smooth bounded domain $\Omega\subset\mathbb{R}^{3}$:
\begin{subequations}
  \label{model2}
  \begin{align}
    \label{model2-1}
   & \partial_t(\rho  \mathbf{u}) + \mathrm{div}\left (\rho  \mathbf{u}\otimes
     \mathbf{u}\right)=  \mathrm{div}\big(S(\phi,\mathbb{ D}\mathbf{u})\big)-(\nabla p
    +\phi \nabla \mu)-\rho  \mathbf{k}, \quad\text{in $Q_{T}$},\\
    \label{model2-2}
    &\mathrm{div}  \mathbf{u} = \alpha \Delta\mu_p, \quad\text{in $Q_{T}$},\\
    \label{model2-3}
    &\partial_t\phi + \mathrm{div}(\phi  \mathbf{u})=\Delta\mu_p, \quad\text{in $Q_{T}$},\\
    \label{model2-4}
   & \mu=f(\phi)-\Delta \phi, \quad\text{in $Q_{T}$},\\
    \label{model2-5}
    &\mu_p=\mu+\alpha p, \; \rho=\frac{\varepsilon}{2}\phi+1+\frac{\varepsilon}{2}, \quad\text{in $Q_{T}$}.
    %\label{model2-6}
%    &\partial_t \rho^{\varepsilon} + \nabla \cdot(\rho^{\varepsilon}\mathbf{u}^{\varepsilon})= 0, \quad\text{in $Q_{T}$}.
  \end{align}
  \end{subequations}
Here $Q_{T}=\Omega\times(0,T)$, the matrix $\mathbf{I}$ denotes the identity matrix and the vector $\mathbf{k}$ stands for the (fixed) unit vector pointing upwards. The unknowns $(\phi,\mathbf{u},p): Q_T \to \mathbb{R} \times \mathbb{R}^3 \times \mathbb{R}$ are the order parameter (volume fraction difference), the mass-averaged (barycentric) mixture velocity and the pressure respectively.
The scalar $\mu$ represents the chemical potential, with $f=F'$, where the function $F(\phi)=\frac{1}{4}(1-\phi^{2})^{2}$ is the potential energy for the interface. Note that the constant $ \varepsilon \coloneqq \frac{\rho_2}{\rho_1} - 1 $ measures how large are the variance of the two densities $ \rho_1, \rho_2 > 0 $ for each pure phase. Thus, without loss of generality, it is assumed to be negative ($ \rho_2 < \rho_1 $). By this, the ratio {$ \alpha \coloneqq -\frac{\rho_2 - \rho_1}{\rho_2 + \rho_1} = - \frac{\varepsilon}{2+\varepsilon} $} is positive.
Moreover, the matrix $S(\phi,\mathbb{ D}\mathbf{u})\coloneqq2\eta(\phi) \mathbb{ D}( \mathbf{u})-\frac{2}{3}\eta(\phi) (\mathrm{div}  \mathbf{u}) \mathbf{I}$ denotes the Newtonian stress tensor with the symmetric gradient $\mathbb{D}(\mathbf{u}) \coloneqq \frac12(\nabla  \mathbf{u}+\nabla  \mathbf{u}^T)$, where
$\eta(\phi)=\frac{\nu-1}{2}\phi+\frac{\nu+1}{2}$ is the viscosity of the mixture and
$\nu>0$ stands for the constant viscosity ratio of the two fluids. Note for simplicity we have set 
the rest of the dimensionaless parameters to be unity. 

Following \cite{HA1} we equip this system with the following  initial-boundary conditions
\begin{subequations}\label{BCs}
\begin{align}
 \label{BCs1}
 &\mathbf{n}\cdot\mathbf{u}=0, \quad \text{on $S_{T}$},  \\&
 \label{BCs2}
 \nabla \phi\cdot \mathbf{n}=\nabla \mu_p \cdot \mathbf{n}=(\mathbf{n}\cdot S(\phi,\mathbb{ D}\mathbf{u}))_{\tau}+(a(\phi)\mathbf{u})_{\tau} =0, \quad \text{on $S_{T}$ }, \\&
 \label{BCs3}
 (\mathbf{u}, \phi)|_{t=0}=( \mathbf{u}_0, \phi_0), \quad \text{in } \Omega,
\end{align}
\end{subequations}
where $S_{T} = \partial \Omega \times (0,T)$ and  $a(\cdot):\mathbb{R}\rightarrow[0,\infty)$ denote the friction parameter on the boundary, $\mathbf{n}$ represent the exterior normal at the boundary of
$\Omega$,  $ (\cdot)_\tau = (\mathbf{I} - \mathbf{n} \otimes \mathbf{n}) \cdot $ denotes the tangent projection on $ \partial \Omega $. Note that {the boundary condition \eqref{BCs2}} for the velocity is the Navier-slip boundary condition.
 Eqs. \eqref{model2-2} and \eqref{model2-3} imply
\begin{align}\label{mass}
	\partial_t \rho+\mathrm{div} (\rho \mathbf{u})=0.
\end{align}
Then by the standard formal test procedure for smooth solutions, i.e., multiplying \eqref{model2-1}, \eqref{model2-3}, and \eqref{model2-4} with $ \mathbf{u} $, $ \mu $, and $ \partial_{t} \phi $ respectively together with integrating by parts, the total energy fulfills
\begin{align}\label{ConEnL}
\frac{\mathrm d}{\mathrm dt}E(\mathbf{u}, \phi)&=
-\int_\Omega \eta(\phi)\left( 2
\mathbb{ D}(\mathbf{u}):\mathbb{ D}(\mathbf{u})-\frac{2}{3} (\mathrm{div} \mathbf{u})^2\right)\, \mathrm{d} x \\
& \quad
-\int_\Omega |\nabla \mu_p|^2\, \mathrm{d} x
-\int_{\partial\Omega} a(\phi)|\mathbf{u}_\tau|^{2}  \, \mathrm{d} x,
\end{align}
where the total energy $E$ is defined as
\begin{align*}%\label{Etot}
  E(\mathbf{u}, \phi)=
  \int_{\Omega}\Big(\frac{1}{2}\rho|\mathbf{u}|^2+\rho  z\Big)\, \mathrm{d} x+
\int_{\Omega} \Big( F(\phi)+\frac{1}{2}|\nabla \phi|^2\Big)\, \mathrm dx,
\end{align*}
with $\nabla z=\mathbf{k}$.

%\cred{\begin{subequations}
%		\label{model3}
%		\begin{align}
%			\label{model3-1}
%			& \partial_t \mathbf{u} + \mathbf{u} \cdot\nabla\mathbf{u}=  \mathrm{div}\big(2\eta(\phi) \mathbb{ D}( \mathbf{u})\big)-(\nabla p
%			+\phi \nabla \mu)-\mathbf{k}, \quad\text{in $Q_{T}$},\\
%			\label{model3-2}
%			&\mathrm{div} \mathbf{u} = 0, \quad\text{in $Q_{T}$},\\
%			\label{model3-3}
%			&\partial_t\phi + \mathrm{div}(\phi  \mathbf{u})=\Delta\mu, \quad\text{in $Q_{T}$},\\
%			\label{model3-4}
%			& \mu=f(\phi)-\Delta \phi, \quad\text{in $Q_{T}$},\\
%			\label{BCs11}
%			&\mathbf{n}\cdot\mathbf{u}=0, \  (\mathbf{n}\cdot S(\phi,\mathbb{ D}\mathbf{u}))_{\tau}+(a(\phi)\mathbf{u})_{\tau} =0, \quad \text{on $S_{T}$}, \\&
%			\label{BCs21}
%			\nabla \phi \cdot \mathbf{n}=\nabla \mu\cdot \mathbf{n} =0, \quad \text{on $S_{T}$}, \\&
%			\label{BCs31}
%			(\mathbf{u}, \phi)|_{t=0}=( \mathbf{u}_0, \phi_0), \quad \text{in } \Omega.
%		\end{align}
%\end{subequations}}
%\yadong{I would delete \eqref{model3} here as it is exactly the same as \eqref{eqs:ModelH}.}

The qCHNS system is a diffuse interface model for two-phase flows of mismatched densities. Formally setting $\alpha=0$, one recovers the celebrated model H for two-phase flows with matched densities, cf. \cite{MDJ, PB}. See \cite{ADGC,GZ} for recent reviews of diffuse interface models for multi-phase flows. There exist several different diffuse interface models for the case of mismatched densities, depending on how the mixture velocity is defined and which order parameter is utilized. {A typical example is the thermodynamically consistent model (with unmatched density) derived  by J. Lowengrub and L. Truskinovsky in \cite{JL}, where the mixture velocity is mass-averaged (conservation of momentum) and the order parameter is the concentration difference. The mass-averaged velocity is no longer divergence-free inside the diffusive interface, hence the model is a quasi-incompressible model. Moreover, the quasi-incompressibility introduces a pressure term in the chemical potential equation, cf.\eqref{model2-5}, which entails a strong coupling between the Navier-Stokes equations and the Cahn-Hilliard equations. Recently,  by using balance laws of the individual constituents and the Coleman--Noll procedure, M. Shokrpour Roudbari et al. \cite{MGEK} derived the qCHNS system \eqref{model2}. The usage of volume fraction difference as the order parameter in \eqref{model2} greatly simplified the quasi-incompressible model, which makes it more amenable for numerical computation. 
	An important observation of M. Shokrpour Roudbari et al. \cite[Remark 2.9 \& 2.10]{MGEK} is that, up to the definition of the mobility and the definition of mass fluxes, the models in \cite{ADGC,SYW,MGEK} are all	 equivalent. A unified derivation and comparison of known diffuse interface
	models from the physical point of view is presented in \cite{MKID}. On the other hand, based on a volume-averaged velocity, H. Abels, H. Garcke and G. Gr\"{u}n \cite{HHG} firstly derived a thermodynamically
consistent diffuse interface model with different densities, termed as the AGG model. Since the mixture velocity is derived from volume averaging, the AGG model is an incompressible model--closely resembling the model H. We also refer the interested readers to \cite{ADGC,FB,HPC,AR,SYW,MKID} for other diffuse interface systems
	modeling two-phase flows with unmatched densities.}

For the AGG model,  the existence of global weak solutions  was established in \cite{HDG} with the logarithmic potential $F$ and a positive mobility $m$, and the case of degenerate mobility was addressed in \cite{HDH}. The local  existence of strong solutions was proved in \cite{HJ} by the Hilbert monotone operator theory and the maximal $ L^p $-regularity theory for fourth-order parabolic equations. Later on, the strong well-posedness and stability result of the model with logarithmic potential
%and constant mobility
in two-dimensional and three-dimensional cases were established by Giorgini with $ L^2 $-energy method \cite{AG,AG1}. Recently, H. Abels, H. Garcke and A. Giorgini \cite{HHA} gave the regularity properties  and long time behavior of weak solutions. There were lots of extensions of this model in different aspects.
Concerning moving contact line problem, C.G. Gal, M. Grasselli and H. Wu \cite{CMH} proved the existence of global weak solutions of above system with dynamic boundary conditions and logarithmic potential.  In \cite{SF1}, S. Frigeri derived a nonlocal version of above model and proved existence of global dissipative weak solutions under the case of singular
double-well potentials and non degenerate mobilities. S. Frigeri later turned from the study of non degenerate mobility to the physically more relevant situation of degenerate mobility and proved existence
of global weak solutions %satisfying an energy inequality
 in \cite{SF2}.  {Recently, the global well-posedness and convergence to equilibrium of the nonlocal model was justified in \cite{GGG}.} In \cite{HD}, weak solutions of a non-Newtonian variant of the AGG model was considered.

The theoretical results on quasi-incompressible diffuse interface models are rather limited. Since the model is quasi-incompressible (compressible), the regularity of pressure is low \cite{HA12}. Note that the pressure enters the chemical potential equation \eqref{model2-5}. {In this case, the pressure does not enjoy any integrability \textit{a priori} and hence one cannot use the standard technique for Cahn--Hilliard equation with elliptic estimates, due to which the logarithmic potential could not be included.} For the Lowengrub--Truskinovsky model, global existence of weak solution is proven in  \cite{HA12} where the energy law yields the {$W^{1,p}, p>d$} estimate of the order parameter.  Local in-time existence of strong solution is established in \cite{HA1}.

The aim of this article is to show local existence of strong solution to the qCHNS model which differs significantly from the Lowengrub-Truskinovsky model. Following the approach in \cite{HA1}, we linearize the system and show that the linear operator is an isomorphism between certain Hilbert spaces, making use of the maximal regularity theory of parabolic systems  and properties of traces of anisotropic fractional Sobolev spaces(cf. \cite{JG}). Then we argue that the nonlinear terms are Lipschitz continuous with a small constants afforded by a small final time $T$. Unlike  the quasi-incompressible model of Lowengrub-Truskinovsky studied by Abels in \cite{HA1}, there is a higher-order term $ \mathrm{div}(\phi \mathbf{u}) $ (precisely $ \phi \mathrm{div} \mathbf{u} $) in presence from the equation of order parameter,  cf. \eqref{div-phi-u}. This higher order term poses substantial difficulty  in current analysis framework. To overcome this challenge, we additionally impose a smallness assumption on the initial data $ \phi_0 $.
The following assumption is maintained throughout.
\begin{assumption}\label{main assumption}
We assume {$\Omega\subset\mathbb{R}^3$ is a smooth bounded domain,} $\eta,a\in C^{2}(\mathbb{R})$, %and $\Omega$ does not have an axis of symmetry,
$0<\inf_{t\in \mathbb{R}}\eta(t)\leq \eta(s)\leq \sup_{t\in \mathbb{R}} \eta(t)<+\infty$, $0<\inf_{t\in \mathbb{R}} a(t)\leq a(s)\leq \sup_{t\in \mathbb{R}} a(t)<+\infty$.
\end{assumption}

The main result is summarized as follows.
\begin{theorem}\label{main theorem1}
Let $\mathbf{u}_{0}\in H^{1}(\Omega)$, and $\mathbf{\phi}_{0}\in H^{2}(\Omega)$ such that  $|\mathbf{\phi}_{0}|<\varepsilon_{0}$ {\text{ a.e.~in }\ $\Omega$ for a sufficiently small} constant $\varepsilon_{0}$. Then there exists $T_{0}>0$ such that for any $0 < T \leq T_{0}$, the system \eqref{model2} equipped with boundary conditions (\ref{BCs}) admits a unique strong solution $ (\mathbf{u}, \phi) $  fulfilling
	\begin{align*}
		&\mathbf{u}\in H^{1}(0,T; L^{2}(\Omega))\cap L^{2}(0,T; H^{2}(\Omega)),
		\\&\phi\in L^{2}(0,T; H^{3}(\Omega))\cap  H^{1}(0,T; H^{1}(\Omega)),
		\\& |\phi|\leq\frac{1}{2},\  \rho\geq\frac{1}{4}.
	\end{align*}
	Moreover, there is a constant $ C > 0 $ such that
	\begin{align*}
		&\|\mathbf{u}\|_{L^{2}(0,T;H^{2}(\Omega))}
		+\|\partial_{t}\mathbf{u}\|_{L^{2}(0,T; L^{2}(\Omega))}
		+\|\phi\|_{L^{2}(0,T; H^{3}(\Omega))}
		+\|\partial_{t}\phi\|_{L^{2}(0,T; H^{1}(\Omega))}
		\\&\leq C(\|\phi_{0}\|_{ H^{2}(\Omega)}+\|\mathbf u_{0}\|_{ H^{1}(\Omega)}).
\end{align*}
\end{theorem}
\begin{Remark}\label{R12}
	For the case $d=2$, the same conclusion can be obtained by applying the same method as $d=3$.
\end{Remark}

\begin{Remark}\label{R123}
Here we get $|\mathbf{\phi}|\leq\frac{1}{2}$ for simplicity. In fact the general result can be written as $|\mathbf{\phi}|<C < 1$ and $\rho\geq C_{0}>0$, cf. Remark \ref{R1}.
\end{Remark}

%\begin{Remark}\label{DI}
%For the case $d=2$, the same conclusion can be obtained by applying the same method as $d=3$.
%\end{Remark}

The structure of this article is as follows. In section \ref{sec:preliminaries}, we recall some classical results and preliminary lemmas.  In order to prove the main result, we shall divide the system into the linear part and nonlinear part to apply the Banach fixed point theorem.  Section \ref{sec:proofofThm} is devoted to the proof of the main result Theorem \ref{main theorem1}. The nonlinear part and main theorem  is set up in this section. In section \ref{sec:proofoflinear}, we deal with the linear part by semigroup theory. Finally, the property of the Stokes operator with Navier boundary conditions is given in Appendix \ref{sec:StokesNavier} and some embedding relationships are given in Appendix \ref{sec:Embedding}.

\section{Preliminaries}
\label{sec:preliminaries}
%The notation $A\les B$ (resp., $A \gtrsim B$) means that there exists a harmless positive constant $c$ such that $A \leq cB$ (resp., $A \geq cB$).
%
\subsection{Notation and Sobolev spaces} For clarity we recall some notations, differential operators and Sobolev spaces in this part.
For $A,B\geq 0$, the notation $A\lesssim B(A \gtrsim B)$ stand for inequalities of the form $A \leq CB(A \geq CB)$ for $C>0$,  $A\approx B$ means $B\lesssim A\lesssim B$. {When the precise value of the constant $C$ or its dependence on other quantities is not required, we simply denote it by C without further specification.}

%Given a Banach space $X$, we denote its norm by $\|\cdot\|_{X}$. For a Banach space $X$ and for any $0<T\leq\infty$, we use standard notation $L^{p}(0,T;X)$ to denote the quasi-Banach space of Bochner measurable functions $f$ from $(0,T)$ to X endowed with the norm.
%\begin{equation*}
% \|f\|_{{L^{p}_{T}}X}:=
% \begin{cases}
%        (\int_0^T \|f(\cdot,t)\|^{p}_{X}\,dt)^{\frac{1}{p}},&\text{if $1\leq p<\infty$},\\
%        \sup_{0\leq t \leq T}\|f(\cdot,t)\|_X,&\text{if $p=\infty$}.
% \end{cases}
% \end{equation*}

 For two vectors  $a,b\in\mathbb{R}^{d}$,  $a\otimes b\triangleq(a_{i}b_{j})_{i,j}^{d}\in\mathbb{R}^{d\times d}$. For two $n\times n$ matrices $A=(a_{ij})_{i,j=1}^d, B=(b_{ij})_{i,j=1}^d$,  $A:B\triangleq\sum_{i,j=1}^{d}a_{ij}b_{ij}$ and
{$\nabla\cdot A\triangleq(\sum_{j=1}^{d}\partial_{x_{j}}a_{i,j})_{i=1}^{d}$}.  If $\mathbf{v}\in C^{1}(\Omega)^{d}$, then $\nabla \mathbf{v}=(\partial_{x_{j}}\mathbf{v}_{i})_{i,j=1}^{d}$.

 %Furthermore, $\mathbf{f}_{n}:=\mathbf{n}\cdot\mathbf{f}$ and $\mathbf{f}_{\tau}:=(\mathbf{I}-\mathbf{n}\otimes\mathbf{n})\mathbf{f}=\mathbf{f}-\mathbf{f}_{n} \mathbf{n}$ denote the normal and tangential component of a vector field $\mathbf{f}:\partial\Omega\rightarrow \mathbb{R}^{d}$, respectively.
 We define $\partial_{n}=\mathbf{n}\cdot\nabla$, $\nabla_{\tau}=(\mathbf{I}-\mathbf{n}\otimes\mathbf{n})\nabla$ and $\partial_{\tau_{j}}=e_{j}\cdot (\mathbf{I}-\mathbf{n}\otimes\mathbf{n})\nabla$, $j=1,\cdot\cdot\cdot,d$,
where $e_{j}$ denotes the $j-th$ canonical unit vector in $\mathbb{R}^{d}$.

 %For a normed linear space $X$, we denote by $\|\cdot\|_X$ the norm on $X$. The duality pairing
%between an abstract vector space $X$ and its dual $X'$ is denoted by $\langle\cdot;\cdot\rangle_{X',X}$. In
%particular, if $X$ is a Hilbert space, the symbol $\langle\cdot;\cdot\rangle$ denotes the scalar product
%in $X$.
 Given a Banach space $X$, $L^{p}(0,T;X)(1\leq p \leq\infty)$ denotes {the space of
strongly measurable $X$-valued functions that are $p$-integrable (essentially bounded if $p=\infty$),
endowed with the norm:}
\begin{equation*}
 \|f\|_{L^{p}(0,T;X)}\triangleq
 \begin{cases}
        (\int_0^T \|f(\cdot,t)\|^{p}_{X}\,\mathrm dt)^{\frac{1}{p}},&\text{if $1\leq p<\infty$},\\
        \sup_{0\leq t \leq T}\|f(\cdot,t)\|_X,&\text{if $p=\infty$}.
 \end{cases}
 \end{equation*}
% Let $\Omega\in \mathbb{R}^{d}$ be any open set, $W^{m,p}(\Omega)$ is the Sobolev space for every $m\in N_{0}$ and $1\leq p\leq\infty$,
% and $W^{m,p}_{0}(\Omega)$ is the closure of $C_{0}^{\infty}(\Omega)$ in $W^{m,p}(\Omega)$. If $p=2$, $W^{m,p}(\Omega)=H^{2}(\Omega)$.
 Moreover, we define for a bounded sufficiently smooth domain $\Omega$
\begin{align*}
&H_{n}^{1}(\Omega)=\{\mathbf{u}\in H^{1}(\Omega)^{d}:\mathbf{n}\cdot\mathbf{u}=0\},
\\& W_{p,N}^{2}(\Omega)=\{u\in W_{p}^{2}(\Omega):\mathbf{n}\cdot \nabla u=0\}.
  \end{align*}
The usual Besov spaces on a bounded domain $\Omega$ are denoted by $B_{p,q}^{s}(\Omega)$, where $s\in \mathbb{R}$, $1\leq p,q\leq\infty$. For convenience of later analysis, we recall(cf \cite{HT})
\begin{align}
 & \label{BB1} B_{p,\infty}^{s+\varepsilon}(\Omega)\hookrightarrow B_{p,q_{1}}^{s}(\Omega)\hookrightarrow B_{p,q_{2}}^{s}(\Omega), \varepsilon>0, 1\leq q_{1}\leq q_{2}\leq\infty,
\\& \label{BB2}B_{p_{1},q}^{s_{1}}(\Omega)\hookrightarrow B_{p_{2},q}^{s_{2}}(\Omega), s_{1} \geq s_{2}, s_{1}-\frac{d}{p_{1}}\geq s_{2}-\frac{d}{p_{2}},
\\& \label{BB3}B_{p,1}^{\frac{d}{p}}(\Omega)\hookrightarrow C_{b}^{0}(\Omega).
  \end{align}

 Let $X $ be a metric space and $E$ be a Banach space, we write $BUC(X,E)$  for the Banach space of all
 bounded and uniformly
continuous functions. Given $\rho \in(0,1]$, we define
$$[u]_{\rho}= \underset{x,y\in X, x\neq y}\sup\frac{\parallel u(x)-u(y)\parallel_{E}}{d(x,y)^{\rho}} $$
for $u: X \rightarrow E$, and
$$C^{\rho}\triangleq\{u\big|\|u\|_{C^{\rho}}=\|u\|_{\infty}+[u]_{\rho}<+\infty\}.$$
%Then
%$$BUC(X, E) := \{ u \in B(X, E) ; [u]_{0} < \infty, \|\cdot\|_{L^{\infty}}<\infty\} .$$

  The $L^{2}$-Bessel potential spaces are denoted by $H^{s}(\Omega)$,$s\in \mathbb{R}$, which are defined by restriction of distributions in $H^{s}(\mathbb{R}^{d})$ to $\Omega$, cf. \cite{HT}. % We denote by $H_{0}^{1}(\Omega)$  the closure of $C_{0}^{\infty}(\Omega)$ with respect to the $H^{1}(\Omega)$ norms.
  %$H_{\sigma}(\Omega)$ and $V_{\sigma}(\Omega)$ is denoted as the closure of $C_{0,\sigma}^{\infty}(\Omega)$ with respect to the $L^{2}(\Omega)$ and $H_{0}^{1}(\Omega)$ norms.
%  Let $\Omega$ a bounded domain with $C^{0,1}$-boundary.
  %Then there is an extension operator $E_{\Omega}$ which is a bounded linear operator $E_{\Omega}: W^{m,p}(\Omega)\rightarrow W^{m,p}(\mathbb{R}^{d})$, $1\leq p\leq\infty$ for all $m \in\mathbb{N}$ and $E_{\Omega}f|_{\Omega} = f$ for all
% $f \in W^{m,p}(\Omega)$, cf. Stein \cite{EMS}. %Hence the following results about the interpolations of $W^{m,p}(\mathbb{R}^{d})$ can be transformed into $W^{m,p}(\Omega)$.  The basic results in interpolation theory refer to Bergh and L\"{o}fstr\"{o}m \cite{JJ}.
 In the following, $(\cdot, \cdot)_{[\theta]}$
and $(\cdot, \cdot)_{\theta,q}$ will denote the complex and real interpolation function, respectively. In particular, we note that
\begin{align}\label{ae}
(H^{s_{0}}(\Omega),H^{s_{1}}(\Omega))_{[\theta]}=(H^{s_{0}}(\Omega),H^{s_{1}}(\Omega))_{\theta,2}=H^{s}(\Omega)
  \end{align}
for $0<\theta<1$ where $s=(1-\theta)s_{0}+\theta s_{1}$, cf. \cite{JJ}.

Finally, $f\in W^{k,p}(0,T; X)(1\leq p \leq\infty, k\in \mathbb{N}_{0})$, if and only if
$f,\cdot\cdot\cdot,\frac{d^{k}f}{dt^{k}}\in  L^{p}(0,T; X)$, where  denotes the $k$-th $X$-valued distributional derivative of $f$. Furthermore,  we set $H^{1}(0,T; X)=W_{2}^{1}(0,T; X)$ and  $H^{s}(0,T;X)=B_{2,2}^{s}(0,T;X)$ for $s\in(0,1)$ where $f\in B_{2,2}^{s}(0,T;X)$ means
\begin{align}
&\|f\|_{B_{2,2}^{s}(0,T;X)}^{2}=\|f\|_{L^{2}(0,T;X)}^{2}+\|f\|_{\dot{H}^{s}(0,T;X)}^{2}<\infty,\label{frcsobolev}
%\\&\|f\|_{\dot{H}^{s}(0,T;X)}^{2}=\|f\|_{\dot{B}_{2,2}^{s}(0,T;X)}^{2}=\int_0^{T}\int_0^{T}\frac{\|f(t)-f(\tau)\|_{X}^{2}}{|t-\tau|^{2s+1}}\mathrm dt \mathrm d\tau<\infty.
  \end{align}
  where
  \begin{align}
%&\|f\|_{B_{2,2}^{s}(0,T;X)}^{2}=\|f\|_{L^{2}(0,T;X)}^{2}+\|f\|_{\dot{H}^{s}(0,T;X)}^{2}<\infty,
\|f\|_{\dot{H}^{s}(0,T;X)}^{2}=\|f\|_{\dot{B}_{2,2}^{s}(0,T;X)}^{2}=\int_0^{T}\int_0^{T}\frac{\|f(t)-f(\tau)\|_{X}^{2}}{|t-\tau|^{2s+1}}\mathrm dt \mathrm d\tau.\label{frcsobolev-1}
  \end{align}

By direct calculations, cf. \cite{HA1}, we have
\begin{align}\label{HC}
\|f\|_{H^{s}(0,T;X)}\leq C_{s,s'}T^{\frac{1}{2}}\|f\|_{C^{s'}(0,T;X)},
  \end{align}
for $f\in C^{s'}(0,T;X)$ if $0<s<s'\leq1$. We also need the following estimates\cite{JS}
\begin{align}\label{HH}
\|f\|_{\dot{H}^{r}(0,T;X)}\leq |T|^{s-r}\|f\|_{\dot{H}^{s}(0,T;X)},
  \end{align}
for $r\leq s$ and $f\in \dot{H}^{s}(0,T;X)$.
Moreover, let $X, X_{0}$ and $X_{1}$ be Banach spaces which satisfy  $\|f\|_{X}\leq \|f\|_{X_{0}}^{1-\theta}\|f\|_{X_{1}}^{\theta}$ for $f\in X_{0}\cap X_{1}$, then there holds the following useful embedding\cite{HA1}
\begin{align}\label{BCC}
BUC([0,T]; X_{1})\cap C^{s}([0,T]; X_{0})\hookrightarrow C^{s(1-\theta)}([0,T]; X), 0<s,\theta<1.
  \end{align}

\subsection{Weak Neumann Laplace equation}
%In the following we assume that
%$\Omega\subset \mathbb{R}^{d}$ is a bounded domain with Lipschitz boundary.
This part is adapted from \cite{HA1}.
%Given $f \in L^{1}(\Omega)$, we denote by $\bar{f}=\frac{1}{|\Omega|}\int\limits_\Omega f(x)\,\mathrm dx$  its mean value. Moreover, for $m\in \mathbb{R}$ we set
%\begin{align*}
%L^{q}_{(m)}(\Omega)=\{f\in L^{q}(\Omega):\bar{f}=m\}, 1\leq q \leq\infty.
%  \end{align*}
Define  the orthogonal projection onto $L^{2}_{(0)}(\Omega)=\{f\in L^{2}(\Omega):\bar{f}\triangleq\frac{1}{|\Omega|}\int\limits_\Omega f(x)\,\mathrm dx=0\}$ by
\begin{align*}
P_{0}f\triangleq f-\bar{f}.
  \end{align*}
Then we define  a Hilbert space $H_{(0)}^{1}(\Omega)$ by
\begin{align*}
H_{(0)}^{1}(\Omega)=H^{1}(\Omega)\cap L^{2}_{(0)}(\Omega)
  \end{align*}
  with the following inner product
 \begin{align*}
(c,d)_{H_{(0)}^{1}(\Omega)}=(\nabla c,\nabla d)_{L^{2}(\Omega)}.
  \end{align*}

%Then $H_{(0)}^{1}(\Omega)$ is due to the Poincar\'{e} inequality
%\begin{align*}
%\|f(t)-\bar{f}\|_{L^{p}(\Omega)}\leq C_{p}\|\nabla f\|_{L^{p}(\Omega)}, 1\leq p\leq\infty.
%  \end{align*}

Moreover, the weak Neumann-Laplace operator $\Delta_{N}:H_{(0)}^{1}(\Omega)\rightarrow H_{(0)}^{-1}(\Omega)$ is defined by
\begin{align*}
-\langle\Delta_{N}u,\varphi\rangle_{H_{(0)}^{-1}(\Omega),H_{(0)}^{1}(\Omega)}=(\nabla u,\nabla \varphi)_{L^{2}(\Omega)}, \quad\text{for $\varphi\in H_{(0)}^{1}(\Omega)$},
  \end{align*}
  where $H_{(0)}^{-1}(\Omega)=H_{(0)}^{1}(\Omega)'$. It follows from Lax--Milgram lemma that for every $f\in H_{(0)}^{-1}(\Omega)$  there is a unique $u\in H_{(0)}^{1}(\Omega)$
such that $\Delta_{N}u=f$ . More precisely, $-\Delta_{N}$ coincides with the Riesz isomorphism $\mathbf{R}:H_{(0)}^{1}(\Omega)\rightarrow H_{(0)}^{-1}(\Omega)$ given by
\begin{align*}
\langle\mathbf{R}c,d\rangle_{H_{(0)}^{-1}(\Omega),H_{(0)}^{1}(\Omega)}=( c, d)_{H_{(0)}^{1}(\Omega)}=(\nabla c,\nabla d)_{L^{2}(\Omega)}, \quad\text{for $c,d\in H_{(0)}^{1}(\Omega)$}.
  \end{align*}

We equip $H_{(0)}^{-1}(\Omega)$ with the inner product
\begin{align}
\langle f,g \rangle_{H_{(0)}^{-1}((\Omega))}=(\nabla \Delta_{N}^{-1}f,\nabla \Delta_{N}^{-1}g)_{L^{2}(\Omega)}=( \Delta_{N}^{-1}f,\Delta_{N}^{-1}g)_{H_{(0)}^{1}(\Omega)}.
  \end{align}
And we embed $L^{2}_{(0)}(\Omega)$ into $H_{(0)}^{-1}(\Omega)$ by defining
\begin{align*}
\langle c,\varphi \rangle_{H_{(0)}^{-1}(\Omega),H_{(0)}^{1}(\Omega)}=\int\limits_\Omega c(x)\varphi(x)\,\mathrm dx, \quad\text{for $\varphi\in H_{(0)}^{1}(\Omega),c\in L^{2}_{(0)}(\Omega)$}.
  \end{align*}
%which gives us a useful inequality
% \begin{align}\label{HHH}
%\|f\|_{L^{2}(\Omega)}^{2}\leq \|f\|_{H_{(0)}^{1}(\Omega)} \|f\|_{H_{(0)}^{-1}(\Omega)},
%  \end{align}
%for all $f\in H_{(0)}^{1}(\Omega)$.
  Then one has
  \begin{align}\label{HLH}
\langle -\Delta_{N} f,g \rangle_{H_{(0)}^{-1}(\Omega)}=(f,g)_{L^{2}(\Omega)},  \quad\text{for $f\in H_{(0)}^{1}(\Omega), g\in L^{2}_{(0)}(\Omega)$}
  \end{align}
  and
  \begin{align}\label{HHH}
\|f\|_{L^{2}(\Omega)}^{2}\leq \|f\|_{H_{(0)}^{1}(\Omega)} \|f\|_{H_{(0)}^{-1}(\Omega)},\quad\text{for $f\in H_{(0)}^{1}(\Omega)$}.
  \end{align}

 Next, we give some  elliptic estimations which will be useful in following sections. Assume $u$ is a solution to the equation
\begin{align}
\Delta_{N}u=f,\ \ f\in H_{(0)}^{-1}(\Omega),\label{NP equation-0}
  \end{align}
  we have
\begin{align*}
\|u\|_{H_{(0)}^{1}(\Omega)}\leq  \|f\|_{H_{(0)}^{-1}(\Omega)}.
  \end{align*}
Furthermore, if $u$ solves
 \begin{align}
\Delta_{N}u=f,\ \ f\in L_{(0)}^{q}(\Omega)(1<q<\infty),\label{NP equation-1}
  \end{align}
  we arrive at $u\in W_{q}^{2}(\Omega)$ and
  \begin{align*}
\Delta u=f,\ \text{a.e.} \ \text{in}\ \Omega,\ {\partial_{n}u=0, \ \text{a.e.} \ \text{on}\ \partial \Omega.}
  \end{align*}
  Moreover, there holds
\begin{align*}
\|u\|_{W_{q}^{k+2}(\Omega)}\leq  C_{q}\|f\|_{W_{q}^{k}(\Omega)},\ f\in W_{q}^{k}(\Omega)\cap L_{(0)}^{q}(\Omega),\ k=0,1,
  \end{align*}
where $C_{q}$ depends only on $q$, $d$ and $\Omega$.

Finally, we further obtain $\Delta_{N}u=\mathrm{div}_{n}\nabla u$ for $u\in  H_{(0)}^{1}(\Omega)$, here
$\mathrm{div}_{n}:L^{2}(\Omega)\rightarrow H_{(0)}^{-1}(\Omega)$ is defined by
\begin{align}
\langle\mathrm{div}_{n} f,\varphi\rangle_{H_{(0)}^{-1}(\Omega),H_{(0)}^{1}(\Omega)}=-(f,\nabla\varphi)_{L^{2}(\Omega)}, \quad\text{for  $f\in L^{2}(\Omega), \varphi\in H_{(0)}^{1}(\Omega)$}.
  \end{align}

%{\bf Positive Operator:} Let $H$ be a Hilbert space and $A\in L(H)$. We say $A$ is positive(respectively, strictly positive), written $A\geq0$(respectively, $A>0$), if and only if for all $\varphi\in H$,
%\begin{align*}
%\langle\varphi,A\varphi>\rangle\geq0, \quad\text{(respectively, $\langle\varphi,A\varphi\rangle >0$ if $\varphi\neq 0$)}\quad.
%  \end{align*}

\subsection{The Helmholtz Decomposition}
\label{sec:helmholtz}
Let $1<p<\infty$ and $\Omega$ be a bounded sufficiently smooth domain. We say that the Helmholtz decomposition for $L^{p}(\Omega)$ exists whenever $L^{p}(\Omega)$ can be decomposed into
\begin{align*}
L^{p}(\Omega)=L_{\sigma}^{p}(\Omega)\oplus G_{p}(\Omega),
  \end{align*}
where
\begin{align*}
 &G_{p}(\Omega):=\{ \mathbf{u}\in L^{p}(\Omega):\mathbf{u}=\nabla\pi  \:\ \text{for some} \:\  \pi\in W_{loc}^{1,p}(\Omega) \},
  \end{align*}
and $L_{\sigma}^{p}(\Omega)$ is the closure of $\{\mathbf{u}\in C_{0}^{\infty}(\Omega)^{d}:\mathrm{div}\mathbf{u}=0\}$. The operator $P_{\sigma}:L^{p}(\Omega)\rightarrow L_{\sigma}^{p}(\Omega)$ having $G_{p}(\Omega)$ as its null space is called the Helmholtz projection. The existence of the Helmholtz projection for $L^{p}(\Omega)$ is very strongly linked with the following weak Neumann problem: given $\mathbf{f}\in L^{p}(\Omega)$, find $\pi\in W^{1,p}(\Omega)$ satisfying
\begin{align*}
 &\int\limits_\Omega(\nabla\pi-\mathbf{f})\cdot \nabla\varphi\,\mathrm dx=0, \varphi\in  W_{(0)}^{1,p}(\Omega).
  \end{align*}
For more details, one refers to \cite{HS,MJY}.

\subsection{Auxiliary Results} {We report} the following two lemmas which comes from \cite{HA1}.
\begin{lemma}\label{jme3}
Let $\Omega\subset\mathbb{R}^{d}$ be a bounded domain with $C^{k}$-boundary, $k\geq2$, $d=2,3$. Then there exist a first order tangential differential operator $A=\sum_{j=1}^{d}a_{j}(x)\partial_{\tau_{j}}$, $a_{j}\in C^{k-1}(\partial\Omega)$,such that
\begin{align*}
&(\mathbf{n}\cdot \nabla^{2}u)_{\tau}|_{\partial\Omega}=\nabla_{\tau}\gamma_{1}u+A\gamma_{0}u,
  \end{align*}
where {$\gamma_{j}u=\partial_{n}^{j}u$ on $\partial \Omega$} and $u\in H^{2}(\Omega)$.
\end{lemma}

\begin{lemma}\label{jme4}
Let $\Omega\subset\mathbb{R}^{d}$ be a bounded domain with $C^{2}$-boundary, for $0<T\leq\infty$, $\nu\in C^{1}(\mathbb{R})$ with $\inf_{s\in\mathbb{R}}\nu(s)>0$, $\phi_0\in H^{2}(\Omega)$, $d=2,3$.

1. There is a bounded linear operator $E:H^{\frac{1}{2}}(\partial\Omega)^{d}\rightarrow H^{2}(\Omega)^{d}$ such that
\begin{align*}
&(\mathbf{n}\cdot \mathbb{ D}E\mathbf{a})_{\tau}|_{\partial\Omega}=\mathbf{a}_{\tau}, E\mathbf{a}|_{\partial\Omega}=0, \mathrm{div} E\mathbf{a}=0
  \end{align*}
for all  $\mathbf{a}\in H^{\frac{1}{2}}(\partial\Omega)^{d}$. Moreover,there is a constant $C>0$ such that $\|E\mathbf{a}\|_{H^1 (\Omega)}\leq C\|\mathbf{a}\|_{H^{-\frac{1}{2}}(\partial\Omega) }$ for all $\mathbf{a}\in H^{\frac{1}{2}}(\partial\Omega)^{d} $.

2. There is a bounded linear operator $E_{T} :H^{\frac{1}{4},\frac{1}{2}}(S_{T})^{d}\rightarrow L^{2}(0,T;H^2 (\Omega))^d \cap H^{1}(0,T;L^2 _{\sigma} (\Omega))^{d}$
such that
\begin{align*}
&(\mathbf{n}\cdot 2\nu(\phi_0)\mathbb{ D}E_{T}\mathbf{a})_{\tau}|_{\partial\Omega}=\mathbf{a}_{\tau}, E_{T}\mathbf{a}|_{\partial\Omega}=0, \mathrm{div} (E_{T}\mathbf{a})=0, E_{T}\mathbf{a}|_{t=0}=0,
  \end{align*}
for all $\mathbf{a}\in H^{\frac{1}{4},\frac{1}{2}}(S_{T})^{d}$. Moreover the operator norm of $E_{T}$ can be estimated independent of $0<T\leq\infty$.

\end{lemma}

We also recall the property of analytic semigroup corresponding to abstract differential system.
\begin{theorem}[\cite{HA1}]\label{hme}
Let $A: D(A) \subset H \rightarrow H$ be a generator of a bounded analytic semigroup on a Hilbert space $H$ and $1 < q < \infty$. Then for every $f\in L^{q}(0,\infty;H)$  and $u_{0}\in (H,D(A))_{1-\frac{1}{q},q}$ there is a unique $u: [0, \infty) \rightarrow H$ such that $\frac{du}{dt}, Au(t)\in L^{q}(0,\infty;H)$ solving
\begin{align*}
&\frac{du}{dt}+Au(t)=f(t),t>0,
\\&u(0)=u_{0}.
  \end{align*}
Moreover, there is a constant $C_{q} > 0$ independent of $f$ and $u_{0}$ such that
\begin{align*}
&\left\|\frac{du}{dt}\right\|_{L^{q}(0,\infty;H)}+\|Au\|_{L^{q}(0,\infty;H)}\leq C_{q}(\|f\|_{L^{q}(0,\infty;H)}+\|u_{0}\|_{(H,D(A))_{1-\frac{1}{q},q}}).
  \end{align*}
\end{theorem}
%\begin{lemma}\label{kme}Gagliardo-Nirenberg-Sobolev inequalities:
%
%(1) If $1\leq p <\infty$
%\begin{align}
%& \|v\|_{L^{p^{*}}}\leq C(p,d)\|\nabla v\|_{L^{p}}, p^{*}=\frac{dp}{d-p},
%\end{align}
%for all $v\in W^{1,p}(\mathbb{R}^{d})$.
%
%(2) If $1\leq s,1\leq p <\infty$  and $sp<n$, then
%\begin{align}
%& \|v\|_{L^{p^{*}}}\lesssim\| v\|_{W^{s,p}}, p\leq q\leq\frac{dp}{d-sp},
%\end{align}
%for all $v\in W^{s,p}(\mathbb{R}^{d})$.
%
%(3) If $1\leq s,1\leq p <\infty$  and $sp>n$, then
%\begin{align}
%& \|v\|_{L^{\infty}}\lesssim\| v\|_{W^{s,p}},
%\end{align}
%for all $v\in W^{s,p}(\mathbb{R}^{d})$.
%\end{lemma}

\section{Proof of Theorem \ref{main theorem1}}
\label{sec:proofofThm}
In this section we prove the main result Theorem \ref{main theorem1}, i.e., the existence and uniqueness of strong solutions to \eqref{model2} equipped with  boundary conditions and initial condition(\ref{BCs}) for a short time. For convenience of computations, we reformulate the system by eliminating the pressure $p$ and the chemical potential $\mu$. Note that
 $$\mathrm{div} \mathbf{u} = \alpha \Delta\mu_p,\ \text{in } \Omega;\  \nabla\mu_p\cdot \mathbf{n}=0,\ \text{on } \partial\Omega.$$
We get $\mu_p=\tfrac{1}{\alpha}G(\mathrm{div}\mathbf{u})$ if $G=\Delta_{N}^{-1}$  is defined  by solving the following weak Neumann Laplace equation
\begin{align*}
\Delta G(g)&=g, \quad \text{in } \Omega,
\\ \partial_{n}G(g)&=0, \quad \text{on } \partial\Omega,
\\ \int_{\Omega}g\mathrm dx&=0,
\end{align*}
where $g=\mathrm{div}\mathbf{u}$. We then  have  $P_{\sigma}\mathbf{u}=\mathbf{u}-\nabla G(g)$.

Hence, we rewrite \eqref{model2} as follows in $Q_{T}$
\begin{subequations}
  \label{model31}
  \begin{align}
    \label{model1-9}
    &\rho\partial_t  \mathbf{u} + \rho\mathbf{u}\cdot\nabla   \mathbf{u}-  \mathrm{div} S(\phi,\mathbb{ D}\mathbf{u})+(\tfrac{1}{\alpha^{2}}\nabla G (\mathrm{div}\mathbf{u})
    +(\phi-\tfrac{1}{\alpha}) \nabla (f(\phi)-\Delta \phi))+ \rho\mathbf{k}=0,\\
    \label{model1-10}
    &\partial_t\phi + \mathrm{div}(\phi  \mathbf{u})=\tfrac{1}{\alpha} \mathrm{div}  \mathbf{u},
  \end{align}
%moreover
%\begin{subequations}
%  \label{model1}
%  \begin{align}
%    \label{model1-91}
%    &\nonumber\alpha^{2}\rho^{\varepsilon}\partial_t  \mathbf{u}^{\varepsilon} + \alpha^{2}\rho^{\varepsilon}\mathbf{u}^{\varepsilon}\cdot\nabla   \mathbf{u}^{\varepsilon}-  \alpha^{2}\nabla\cdot S(\phi^{\varepsilon},\mathbf{u}^{\varepsilon})
%    \\&+(\nabla\Delta^{-1} \nabla \cdot  \mathbf{u}^{\varepsilon}
%    +(\alpha^{2}\phi^{\varepsilon}-\alpha) \nabla (f(\phi^{\varepsilon})-\Delta \phi^{\varepsilon}))-  \alpha^{2}\rho^{\varepsilon}\mathbf{k}=0,\\
%    \label{model1-101}
%    &\partial_t\phi^{\varepsilon} + \nabla\cdot(\phi^{\varepsilon}  \mathbf{u}^{\varepsilon})=\tfrac{1}{\alpha} \nabla \cdot  \mathbf{u}^{\varepsilon},
%  \end{align}
%\end{subequations}
with the following boundary conditions and initial condition
\begin{align}
 \label{BCs3111}
 &\mathbf{u}\cdot \mathbf{n}=0, \quad \text{on } \partial \Omega \times (0,T), \\&
 \label{BCs312}
 \nabla \phi \cdot \mathbf{n}=\nabla G (\mathrm{div} \mathbf{u}) \cdot \mathbf{n} =0, \quad \text{on } \partial \Omega \times (0,T),
 \\&
 \label{BCs321}
(\mathbf{n}\cdot S(\phi,\mathbb{ D}\mathbf{u}))_{\tau}+(a(\phi)\mathbf{u})_{\tau}=0, \quad \text{on } \partial \Omega \times (0,T), \\&
 \label{BCs333}
 (\mathbf{u}, \phi)|_{t=0}=( \mathbf{u}_0, \phi_0), \quad \text{in } \Omega.
\end{align}
\end{subequations}

Before the main argument, we reduce the initial data to be vanishing via temporal trace operator. Namely, one shall choose $\tilde{\phi}\in H^{1}(0,T;H^{1}(\Omega))\cap L^{2}(0,T;H^{3}(\Omega)\cap H_{N}^{2}(\Omega))$ be such that $\tilde{\phi}|_{t=0}=\phi_{0}$. %We may assume
%$\rho^{\varepsilon}>\delta_{0}$ under $[0,T_{1}]$ where $\delta_{0}$ is a relatively small positive constant.

%\yadong{Here the logic is: we reformulate the system with modification of $ \mathrm{div} (\phi_0 \mathbf{u}) $. Then prove the system without $ \mathrm{div} (\phi_0 \mathbf{u}) $ in next section. In the beginning of the next section, we state that the operator can be proved in a way without this term. Also in Proposition \ref{pjme} and \ref{pjme1}, just state them with operators defined here.}

To this end, we rewrite the system \eqref{model31} as follows
\begin{equation}
	L(\phi)\binom{\mathbf{u}}{\phi'}=F(\mathbf{u},\phi),\label{main model-00}
\end{equation}
where the operator $L(\phi):X_{T}\rightarrow Y_{T}$  is defined by
\begin{equation}\label{main model}
L(\phi)\binom{\mathbf{u}}{\phi'}=
 \left(
 \begin{array}{c}
\partial_t  \mathbf{u} -\mathrm{div}(\tfrac{1}{\rho_{0}}S(\phi,\mathbb{ D}\mathbf{u}))+h(\phi,\phi')
    \\
\partial_t \phi'-\tfrac{1}{\alpha} \mathrm{div}  \mathbf{u}  %+ \mathrm{div}  (\phi_0\mathbf{u})
\\
(\mathbf{n}\cdot S(\phi,\mathbb{ D}(P_{\sigma}\mathbf{u})))_{\tau}+(a(\phi)P_{\sigma}\mathbf{u})_{\tau}\big|_{\partial\Omega}\\
(\mathbf{u},\phi')|_{t=0} \\
\end{array}
\right),
\end{equation}
%where $S(\phi^{\varepsilon},\mathbf{u}^{\varepsilon})=\nabla\cdot\big(2\eta(\phi^{\varepsilon}) \mathbb{ D}( \mathbf{u}^{\varepsilon})-\frac{2}{3}\eta(\phi^{\varepsilon}) (\nabla \cdot  \mathbf{u}^{\varepsilon}) \mathbf{I}\big)$
for $h(\phi,\phi'):=\nabla(\mathrm{div} (\tfrac{1}{\rho_{0}}(\tfrac{1}{\alpha}-\phi)\nabla \phi'))$ and $\phi'=\phi- \tilde{\phi}$. Moreover, we define the nonlinear part  $F(\mathbf{u},\phi):X_{T}\rightarrow Y_{T}$ as
\begin{equation*}
F(\mathbf{u},\phi)=
\left
(\begin{array}{c}
F_{1}(\mathbf{u},\phi)\\
F_{2}(\mathbf{u},\phi)\\
F_{3}(\mathbf{u},\phi)\\
%-\partial_t \tilde{\phi}-\mathrm{div}((\phi-\phi_0)  \mathbf{u})
%\\
%-(\mathbf{n}\cdot S(\phi,\nabla^{2} G(\mathrm{div}\mathbf{u}))_{\tau}-(a(\phi)\nabla G(\mathrm{div}\mathbf{u})))_{\tau}\big|_{\partial\Omega}\\
(\mathbf{u}_{0},\phi'_{0}) \\
\end{array}
\right),
\end{equation*}
where
\begin{align*}
	F_{1}(\mathbf{u},\phi)&=-\mathbf{u}\cdot\nabla   \mathbf{u}
	- \frac{1}{\rho(\phi)}(\phi-\tfrac{1}{\alpha}) \nabla (f(\phi))-\tfrac{1}{\alpha^{2}}
	\frac{1}{\rho}\nabla G( \mathrm{div}  \mathbf{u}) +\mathbf{k}+\frac{1}{\rho}\nabla(\nabla(\tfrac{1}{\alpha}-\phi)\cdot\nabla\phi)
	\\&\quad
	-\nabla\frac{1}{\rho} \cdot S(\phi,\mathbb{ D}\mathbf{u})+\frac{1}{\rho}\nabla(\tfrac{1}{\alpha}-\phi) \Delta \phi
	+\nabla(\nabla\frac{1}{\rho}\cdot ((\tfrac{1}{\alpha}-\phi)\nabla \phi'))
	\\&\quad
	-(\nabla\frac{1}{\rho} )\mathrm{div}((\phi-\tfrac{1}{\alpha})\nabla \phi))+\nabla(\frac{1}{\rho} \mathrm{div}((\phi-\tfrac{1}{\alpha})\nabla \tilde{\phi})))
	- \nabla\mathrm{div} ((\frac{1}{\rho}-\tfrac{1}{\rho_{0}})(\tfrac{1}{\alpha}-\phi)\nabla \phi'))
	\\&\quad+\mathrm{div}((\frac{1}{\rho}-\tfrac{1}{\rho_{0}})S(\phi,\mathbb{ D}\mathbf{u})),\\
	F_{2}(\mathbf{u},\phi)&=-\partial_t \tilde{\phi}-\mathrm{div}(\phi  \mathbf{u}),\\
	F_{3}(\mathbf{u},\phi)&=
	-(\mathbf{n}\cdot S(\phi,\nabla^{2} G(\mathrm{div}\mathbf{u}))_{\tau}-(a(\phi)\nabla G(\mathrm{div}\mathbf{u})))_{\tau}\big|_{\partial\Omega}.
\end{align*}
%\begin{align*}
%   &F_{1}(\mathbf{u}^{\varepsilon},\phi^{\varepsilon})=(1-\rho^{\varepsilon})\partial_t  \mathbf{u}^{\varepsilon}-\rho^{\varepsilon}\mathbf{u}^{\varepsilon}\cdot\nabla   \mathbf{u}^{\varepsilon}-\tfrac{1}{\alpha^{2}}\nabla\Delta^{-1} \nabla \cdot  \mathbf{u}^{\varepsilon}
%  \\&  +(\phi^{\varepsilon}-\tfrac{1}{\alpha}) \nabla (f(\phi^{\varepsilon}))+ \rho^{\varepsilon}\mathbf{k}-\nabla\nabla\cdot ((\tfrac{1}{\alpha})\nabla \phi'^{\varepsilon}))
%  \\&+\nabla\nabla\cdot ((\tfrac{1}{\alpha}-\phi^{\varepsilon})\nabla {\phi}^{\varepsilon}))+\nabla(\nabla(\tfrac{1}{\alpha}-\phi^{\varepsilon})\cdot\nabla\phi^{\varepsilon}).
%  \end{align*}
Here the spaces are defined as $X_{T}=X_{T}^{1}\times X_{T}^{2}$
%and $X_{T}^{0}:=\{\mathbf{u}\in X_{T}^{1},\phi\in X_{T}^{2},\mathbf{a}\in H^{\frac{1}{4},\frac{1}{2}}(S_{T}), \text{with } \phi_{0}=0\}$
and
\begin{align*}
   &X_{T}^{1}=\{\mathbf{u}\in H^{1}(0,T;L^{2}(\Omega))\cap  L^{2}(0,T;H^{2}(\Omega)):\mathbf{u}|_{\partial\Omega}=0\} ,\\&
   X_{T}^{2}=\{\phi\in H^{1}(0,T;H^{1}(\Omega))\cap  L^{2}(0,T;H^{3}(\Omega)): \mathbf{n}\cdot\nabla\phi|_{\partial\Omega}=0\},\\&
  Y_{T}=L^{2}(Q_{T})\times L^{2}(0,T;H^{1}_{(0)}(\Omega))\times\{\mathbf{a}\in H^{\frac{1}{4},\frac{1}{2}}(S_{T}):\mathbf{n}\cdot\mathbf{a}=0\}\times H^{1}_{n}(\Omega)\times H^{2}_{N}(\Omega),
  \end{align*}
with $H^{\frac{1}{4},\frac{1}{2}}(S_{T})=H^{\frac{1}{4}}(0,T;L^{2}(\partial\Omega))\cap L^{2}(0,T;H^{\frac{1}{2}}(\partial\Omega))$. The norms of above spaces are given by
\begin{align*}
   &\|\mathbf{u}\|_{X_{T}^{1}}=\|(\partial_{t}u,\mathbf{u},\nabla\mathbf{u},\nabla^{2}\mathbf{u})\|_{L^{2}(Q_{T})} ,\\&
  \|\phi\|_{X_{T}^{2}} =\|(\phi,\partial_{t}\phi,\nabla \phi,\nabla^{2}\phi,\nabla^{3}\phi)\|_{L^{2}(Q_{T})},\\&
  \|(f,g,\mathbf{a},\mathbf{u}_{0},\phi_{0}')\|_{Y_{T}}=\|(f,\nabla g)\|_{L^{2}(Q_{T})}+\|\mathbf{a}\|_{H^{\frac{1}{4},\frac{1}{2}}(S_{T})}+\|\mathbf u_{0}\|_{H^{1}}+\|\phi_{0}'\|_{H^{2}}.
  \end{align*}

Now we begin to construct strong solutions to the linearized system \eqref{main model-00}, and prove the corresponding operator  is isomorphism between the proper spaces. By means of the contraction mapping principle, the proof of uniqueness and existence can be obtained for a small time $T$.
%\begin{theorem}\label{ime}
%Let $A: D(A) \subset H \rightarrow H$ be a generator of a bounded analytic semigroup on a Hilbert space $H$ and let $1 < q < \infty$. Then for every $f\in L^{q}(0,\infty;H)$  and $\mathbf{u}_{0}\in (H,D(A))_{1-\frac{1}{q},q}$ there is a unique $\mathbf{u}: [0, \infty) \rightarrow H$ such that $\frac{d\mathbf{u}}{dt}, A\mathbf{u}(t)\in L^{q}(0,\infty;H)$ solving
%\begin{align*}
%&\frac{d\mathbf{u}}{dt}+A\mathbf{u}(t)=f(t),t>0,
%\\&\mathbf{u}(0)=\mathbf{u}_{0}.
%  \end{align*}
%Moreover, there is a constant $C_{q} > 0$ independent of $f$ and $u_{0}$ such that
%\begin{align*}
%&\|\frac{d\mathbf{u}}{dt}\|_{L^{q}(0,\infty;H)}+\|A\mathbf{u}(t)\|_{L^{q}(0,\infty;H)}\leq C_{q}(\|f(t)\|_{L^{q}(0,\infty;H)}+\|\mathbf{u}_{0}\|_{(H,D(A))_{1-\frac{1}{q},q}}).
%  \end{align*}
%\end{theorem}
\begin{Remark}\label{R1}
Before our main proposition, we first give an important fact about $\rho$ throughout this paper. By
Lemma {\ref{jme2}}, one have $\|\phi-\phi_{0}\|_{ L^{\infty}}\leq T^{\frac{1}{8}}\|\phi-\phi_{0}\|_{ X_{T}^{2}}\leq 2T^{\frac{1}{8}}R$ provided $\|\phi\|_{X^{2}_{T}}\leq R$. Thus, choosing $ T>0 $ sufficiently small, one gets $|\phi|\leq\frac{1}{2}$, which implies $\rho\geq\frac{1}{4}$ by (\ref{model2-5}).
\end{Remark}

The following proposition will be  used later in the proof of Theorem \ref{main theorem1}.
\begin{proposition}\label{pjme1}
If  $\phi_{0}\in H^{2}(\Omega)$, $T_{0}>0$  and Assumption {\ref{main assumption}} holds, then $L(\phi_{0}):X_{T}\rightarrow Y_{T}$ is an isomorphism for every
$T\in(0,T_{0}]$ and
\begin{align*}
   \|\big(L(\phi_{0})\big)^{-1}\|_{\mathcal L( Y_{T}, X_{T})}\leq C(T_{0}),
  \end{align*}
where $C(T_{0})$ is a constant dependent of $T_{0}$.  %The following estimates also holds by section $4$:
%\begin{align*}
%   \|\widetilde{L}^{-1}(\phi_{0})\|_{ Y_{T}, X_{T}}\leq C(T_{0}),
%  \end{align*}
\end{proposition}
We will postpone the proof of   Proposition \ref{pjme1} in section \ref{sec:proofoflinear}.

\begin{proposition}\label{pjme}
Let $(\mathbf{u}_{j},\phi_{j})\in X_{T}$ and  $\|(\mathbf{u}_{j},\phi_{j})\|_{X_{T}}\leq R$ for some $R>0$, $\phi_{j}|_{t=0}=\phi_{0}$ and $|\phi_{0}|\leq \varepsilon_{0}$ {\text{ a.e.~in } $\Omega$}, $j=1,2$, and Assumption {\ref{main assumption}} hold. Then there exists $C(T,R)>0$ such that
\begin{align*}
   \|F(\mathbf{u}_{1},\phi_{1})-F(\mathbf{u}_{2},\phi_{2})\|_{ Y_{T}}\leq (C(T,R)+C\varepsilon_{0})\|(\mathbf{u}_{1}-\mathbf{u}_{2},\phi_{1}-\phi_{2})\|_{ X_{T}},
  \end{align*}
where $C(T,R)\rightarrow0$ as $T\rightarrow0$. %Moreover,
%\begin{align*}
%   \|\widetilde{F}(\mathbf{u}_{1},\phi_{1})-\widetilde{F}(\mathbf{u}_{2},\phi_{2})\|_{ Y_{T}}\leq C(T,R)\|(\mathbf{u}_{1}-\mathbf{u}_{2},\phi_{1}-\phi_{2})\|_{ X_{T}},
%  \end{align*}
%where $\widetilde{F}$ is defined in section $4$.
\end{proposition}
\begin{proof}
%Let $F_{1}(\mathbf{u}^{\varepsilon},\phi^{\varepsilon})=(1-\rho^{\varepsilon})\partial_t  \mathbf{u}^{\varepsilon}-\rho^{\varepsilon}\mathbf{u}^{\varepsilon}\cdot\nabla   \mathbf{u}^{\varepsilon}+(\tfrac{1}{\alpha^{2}}\nabla\Delta^{-1} \nabla \cdot  \mathbf{u}^{\varepsilon}
%    +(\phi^{\varepsilon}-\tfrac{1}{\alpha}) \nabla (f(\phi^{\varepsilon})-\Delta \tilde{\phi}^{\varepsilon}))+ \rho^{\varepsilon}\mathbf{k}$
%and $F_{2}(\mathbf{u}^{\varepsilon},\phi^{\varepsilon})=\partial_t \tilde{\phi}^{\varepsilon}-\nabla\cdot(\phi^{\varepsilon}  \mathbf{u}^{\varepsilon})$.

%We first note a basic argument that will used later£»
%\begin{align*}
%   \|f\|_{ L^{2}(0,T;X)}\leq T^{\frac{1}{2}-\frac{1}{p}}\|f\|_{ L^{p}(0,T;X)},
%  \end{align*}
%  where $X$ is a Banach space.

Firstly, we frequently use the following estimate
 \begin{align*}
   \|\phi\|_{ L^{\infty}(0,T;H^{2}(\Omega))}\leq C\|\phi\|_{ H^{1}(0,T;H^{1}(\Omega))}^{\frac{1}{2}}\|\phi\|_{L^{2}(0,T;H^{3}(\Omega))}^{\frac{1}{2}}\leq C\|\phi\|_{ X_{T}^{2}},
  \end{align*}
which is due to  {(\ref{ae})} and Lemma {\ref{jme}}, $C$ is a constant independent of $T$. Since $H^{2}(\Omega)$ is  supercritical $(H^{2}(\Omega)\hookrightarrow L^{\infty}(\Omega))$, the following result  also holds if $\phi_{1},\phi_{2}\in X_{T}^{2}$ and $f\in C^{3}$
\begin{align}
   \|f(\phi_{1})-f(\phi_{2})\|_{ L^{\infty}(0,T;L^{\infty}(\Omega))}\leq\|f(\phi_{1})-f(\phi_{2})\|_{ L^{\infty}(0,T;H^{2}(\Omega))}\leq C\|\phi_{1}-\phi_{2}\|_{ X_{T}^{2}}.
  \end{align}

We divide the proof of this proposition into three parts to proceed.

{\bf Estimation of $F_1(\mathbf{u},\phi)$.}
Starting with H\"{o}lder's inequality, using  Lemma {\ref{jme}} and Sobolev embedding theorem, we have
\begin{align*}
   &\|\mathbf{u}_{1}\cdot\nabla   \mathbf{u}_{1}-\mathbf{u}_{2}\cdot\nabla   \mathbf{u}_{2}\|_{ L^{2}(0,T;L^{2}(\Omega))}
   \\&=\|\mathbf{u}_{1}\cdot\nabla   \mathbf{u}_{1}-\mathbf{u}_{2}\cdot\nabla   \mathbf{u}_{1}+\mathbf{u}_{2}\cdot\nabla   \mathbf{u}_{1}-\mathbf{u}_{2}\cdot\nabla   \mathbf{u}_{2}\|_{ L^{2}(0,T;L^{2}(\Omega))}
    \\&\leq\|  \mathbf{u}_{1}-  \mathbf{u}_{2}\|_{ L^{\infty}(0,T;L^{6}(\Omega))}\|\nabla\mathbf{u}_{1}\|_{ L^{2}(0,T;L^{3}(\Omega))}
    +\| \mathbf{u}_{2}\|_{ L^{\infty}(0,T;L^{6}(\Omega))}{\|\nabla \mathbf{u}_{1}-\nabla  \mathbf{u}_{2}\|_{ L^{2}(0,T;L^{3}(\Omega))}}
     \\&\leq CT^{\frac{1}{8}}\|  \mathbf{u}_{1}-  \mathbf{u}_{2}\|_{ L^{\infty}(0,T;H^{1}(\Omega))}\|\nabla\mathbf{u}_{1}\|_{L^{4}(0,T;L^{2}(\Omega))}^{\frac{1}{2}}
     \|\nabla\mathbf{u}_{1}\|_{L^{2}(0,T;L^{6}(\Omega))} ^{\frac{1}{2}} \\&\quad+CT^{\frac{1}{4}}\| \mathbf{u}_{2}\|_{ L^{\infty}(0,T;H^{1}(\Omega))}\|\nabla \mathbf{u}_{1}-\nabla  \mathbf{u}_{2}\|_{ L^{\infty}(0,T;L^{2}(\Omega))}^{\frac{1}{2}}
     \|\nabla \mathbf{u}_{1}-\nabla  \mathbf{u}_{2}\|_{ L^{2}(0,T;L^{6}(\Omega))}^{\frac{1}{2}}
    \\ &\leq C(R)T^{\frac{1}{4}}\|\mathbf{u}_{1}-  \mathbf{u}_{2}\|_{ X_{T}^{1}}.
  \end{align*}
%\begin{align*}
%   &\|\rho^{\varepsilon}(\phi_{1}^{\varepsilon})\mathbf{u}_{1}^{\varepsilon}\cdot\nabla   \mathbf{u}_{1}^{\varepsilon}-\rho^{\varepsilon}(\phi_{2}^{\varepsilon})\mathbf{u}_{2}^{\varepsilon}\cdot\nabla   \mathbf{u}_{2}^{\varepsilon}\|_{ L^{\infty}(0,T;L^{2})}
%   \\&=\|\rho^{\varepsilon}(\phi_{1}^{\varepsilon})\mathbf{u}_{1}^{\varepsilon}\cdot\nabla   \mathbf{u}_{1}^{\varepsilon}-\rho^{\varepsilon}(\phi_{2}^{\varepsilon})\mathbf{u}_{1}^{\varepsilon}\cdot\nabla   \mathbf{u}_{1}^{\varepsilon}+\rho^{\varepsilon}(\phi_{2}^{\varepsilon})\mathbf{u}_{1}^{\varepsilon}\cdot\nabla   \mathbf{u}_{1}^{\varepsilon}
%    \\&-\rho^{\varepsilon}(\phi_{2}^{\varepsilon})\mathbf{u}_{2}^{\varepsilon}\cdot\nabla   \mathbf{u}_{1}^{\varepsilon}+\rho^{\varepsilon}(\phi_{2}^{\varepsilon})\mathbf{u}_{2}^{\varepsilon}\cdot\nabla   \mathbf{u}_{1}^{\varepsilon}-\rho^{\varepsilon}(\phi_{2}^{\varepsilon})\mathbf{u}_{2}^{\varepsilon}\cdot\nabla   \mathbf{u}_{2}^{\varepsilon}\|_{ L^{\infty}(0,T;L^{2})}
%   \\&\leq C(R)\|\phi_{1}-\phi_{2}\|_{ X_{T}^{2}}.
%  \end{align*}
Next, by Remark \ref{R1} and Lemma {\ref{jme}}, we can show
\begin{align*}
   &\quad\|\tfrac{1}{\rho(\phi_{1})}\tfrac{1}{\alpha^{2}}\nabla G( \mathrm{div}\mathbf{u}_{1})-\tfrac{1}{\rho(\phi_{2})}\tfrac{1}{\alpha^{2}}\nabla G( \mathrm{div} \mathbf{u}_{2})\|_{ L^{2}(0,T;L^{2}(\Omega))}
    \\&\quad\leq CT^{\frac{1}{2}}\|\tfrac{1}{\rho(\phi_{1})}-\tfrac{1}{\rho(\phi_{2})}\|_{ L^{\infty}(0,T;L^{6}(\Omega))}
    \|\nabla G(\mathrm{div}  \mathbf{u}_{1})\|_{ L^{\infty}(0,T;L^{3}(\Omega))}
    \\&\qquad+CT^{\frac{1}{2}}\|\tfrac{1}{\rho(\phi_{2})}\|_{ L^{\infty}(0,T;L^{6}(\Omega))}\|\nabla G(\mathrm{div} \mathbf{u}_{1})-\nabla G( \mathrm{div}  \mathbf{u}_{2})\|_{ L^{\infty}(0,T;L^{3}(\Omega))}
   \\&\quad\leq C(R)T^{\frac{1}{2}}(\|  \phi_{1}-  \phi_{2}\|_{X_{T}^{2}}+\|  \mathbf{u}_{1}-  \mathbf{u}_{2}\|_{ X_{T}^{1}}).
  \end{align*}
By Remark \ref{R1} and Lemma \ref{jme2},
%Next we claim that
%\begin{align*}
%   &\|P_{\sigma}((\rho_{0}-\rho(\phi_{1}))\partial_t  u_{1})-P_{\sigma}((\rho_{0}-\rho(\phi_{2}))\partial_t  u_{2})\|_{ L^{\infty}(0,T;L^{2})}
%  \\& \leq\|((\rho_{0}-\rho(\phi_{1}))\partial_t  u_{1})-((\rho_{0}-\rho(\phi_{2}))\partial_t  u_{2})\|_{ L^{\infty}(0,T;L^{2})}
%  \\& \leq\|((\rho_{0}-\rho(\phi_{1}))\partial_t  u_{1})-((\rho_{0}-\rho(\phi_{2}))\partial_t  u_{1})\|_{ L^{\infty}(0,T;L^{2})}
% \\& +\|((\rho_{0}-\rho(\phi_{2}))\partial_t  u_{1})-((\rho_{0}-\rho(\phi_{2}))\partial_t  u_{2})\|_{ L^{\infty}(0,T;L^{2})}
% \\& \leq\|((\rho(\phi_{2})-\rho(\phi_{1}))\partial_t  u_{1})\|_{ L^{\infty}(0,T;L^{2})}
%  +\|((\rho_{0}-\rho(\phi_{2}))(\partial_t  u_{1}-\partial_t  u_{2}))\|_{ L^{\infty}(0,T;L^{2})}
%  \\& \leq\|(\rho(\phi_{2})-\rho(\phi_{1}))\|_{ L^{\infty}(0,T;H^{2})}\|\partial_t  u_{1}\|_{ L^{\infty}(0,T;H^{2})}
%  \\&+\|(\rho_{0}-\rho(\phi_{2}))\|_{ L^{\infty}(0,T;H^{2})}\|(\partial_t  u_{1}-\partial_t  u_{2})\|_{ L^{\infty}(0,T;H^{2})}
%  \\& \leq C(R)(\|\phi_{2}-\phi_{1}\|_{ X_{T}^{2}}+\|(u_{2}-u_{1})\|_{ X_{T}^{1}}).
%  \end{align*}
  one similarly deduces
\begin{align*}
   &\|\tfrac{1}{\rho(\phi_{1})}(\phi_{1}-\tfrac{1}{\alpha}) \nabla (f(\phi_{1}))-\tfrac{1}{\rho(\phi_{2})}(\phi_{2}-\tfrac{1}{\alpha}) \nabla (f(\phi_{2}))\|_{ L^{2}(0,T;L^{2}(\Omega))}
   \\&= \|\tfrac{1}{\rho(\phi_{1})}(\phi_{1}-\tfrac{1}{\alpha}) \nabla (f(\phi_{1}))-\tfrac{1}{\rho(\phi_{2})}(\phi_{1}-\tfrac{1}{\alpha}) \nabla (f(\phi_{1}))
   \\&\quad+\tfrac{1}{\rho(\phi_{2})}(\phi_{1}-\tfrac{1}{\alpha}) \nabla (f(\phi_{1}))-\tfrac{1}{\rho(\phi_{2})}(\phi_{2}-\tfrac{1}{\alpha}) \nabla (f(\phi_{1}))
   \\&\quad+\tfrac{1}{\rho(\phi_{2})}(\phi_{2}-\tfrac{1}{\alpha}) \nabla (f(\phi_{1}))-\tfrac{1}{\rho(\phi_{2})}(\phi_{2}-\tfrac{1}{\alpha}) \nabla (f(\phi_{2}))\|_{ L^{2}(0,T;L^{2}(\Omega))}
   \\&\leq CT^{\frac{1}{2}}\big(\|\tfrac{1}{\rho(\phi_{1})}-\tfrac{1}{\rho(\phi_{2})} \|_{ L^{\infty}(0,T;L^{\infty}(\Omega))}\|(
   \phi_{1}-\tfrac{1}{\alpha})\|_{ L^{\infty}(0,T;L^{\infty}(\Omega))}\|\nabla (f(\phi_{1}))\|_{ L^{\infty}(0,T;L^{2}(\Omega))}
   \\&\quad+\|\tfrac{1}{\rho(\phi_{2})} \|_{ L^{\infty}(0,T;L^{\infty}(\Omega))}\|(
   \phi_{1}-\phi_{2})\|_{ L^{\infty}(0,T;L^{\infty}(\Omega))}\|\nabla (f(\phi_{1}))\|_{ L^{\infty}(0,T;L^{2}(\Omega))}
   \\&\quad+\|\tfrac{1}{\rho(\phi_{2})} \|_{ L^{\infty}(0,T;L^{6}(\Omega))}\|(
   \phi_{2}-\tfrac{1}{\alpha})\|_{ L^{\infty}(0,T;L^{\infty}(\Omega))}\|\nabla (f(\phi_{1})-f(\phi_{2}))\|_{ L^{\infty}(0,T;L^{3}(\Omega))}\big)
   \\&\leq C(R)T^{\frac{1}{2}}\|\phi_{1}-  \phi_{2}\|_{ X_{T}^{2}}.
  \end{align*}

In view of Lemma {\ref{jme}}, Lemma {\ref{jme1}}, $(H^{2}(\Omega),H^{3}(\Omega))_{\frac{1}{2},1}=B^{\frac{5}{2}}_{2,1}(\Omega)$(cf. \cite{JJ}) and Sobolev embedding,  we infer
\begin{align*}
   &\|\tfrac{1}{\rho(\phi_{1})}\nabla(\nabla\phi_{1}\cdot\nabla\phi_{1})
   -\tfrac{1}{\rho(\phi_{2})}\nabla(\nabla\phi_{2}\cdot\nabla\phi_{2})\|_{ L^{2}(0,T;L^{2}(\Omega))}
  \\&= \|\tfrac{1}{\rho(\phi_{1})}\nabla(\nabla(\phi_{1})\cdot\nabla\phi_{1})
  -\tfrac{1}{\rho(\phi_{2})}\nabla(\nabla\phi_{1}\cdot\nabla\phi_{1})
  \\&\quad+\tfrac{1}{\rho(\phi_{2})}\nabla(\nabla\phi_{1}\cdot\nabla\phi_{1})
   -\tfrac{1}{\rho(\phi_{2})}\nabla(\nabla\phi_{2}\cdot\nabla\phi_{1})
   \\&\quad+\tfrac{1}{\rho(\phi_{2})}\nabla(\nabla\phi_{2}\cdot\nabla\phi_{1})
   -\tfrac{1}{\rho(\phi_{2})}\nabla(\nabla\phi_{2}\cdot\nabla\phi_{2})\|_{ L^{2}(0,T;L^{2}(\Omega))}
   \\&\leq\|\tfrac{1}{\rho(\phi_{1})}-\tfrac{1}{\rho(\phi_{2})}\|_{ L^{\infty}(0,T;L^{\infty}(\Omega))}\|\nabla\phi_{1}\cdot\nabla\phi_{1}\|_{ L^{2}(0,T;H^{1}(\Omega))}
   \\&\quad+\|\tfrac{1}{\rho(\phi_{2})}\|_{ L^{\infty}(0,T;L^{\infty}(\Omega))}
   \|\nabla(\phi_{1}-\phi_{2})\cdot\nabla\phi_{1}\|_{ L^{2}(0,T;H^{1}(\Omega))}
   \\&\quad+\|\tfrac{1}{\rho(\phi_{2})}\|_{ L^{\infty}(0,T;L^{\infty}(\Omega))}
   \|\nabla\phi_{2}\cdot\nabla(\phi_{1}-\phi_{2})\|_{ L^{2}(0,T;H^{1}(\Omega))}
     \\&\leq  CT^{\frac{1}{4}}R
     \Big(
     \|\phi_{1}-\phi_{2}\|_{
     L^{\infty}(0,T;L^{\infty})}\|\phi_{1}\|_{ L^{\infty}(0,T;H^{2}(\Omega))} \|\phi_{1}\|_{ L^{4}(0,T;B^{\frac{5}{2}}_{2,1}(\Omega))}
     \\&\qquad\qquad+{\|\phi_{1}-\phi_{2}\|_{ L^{\infty}(0,T;H^{2}(\Omega))}}
     \|\phi_{1}\|_{ L^{4}(0,T;B^{\frac{5}{2}}_{2,1}(\Omega))}
     \\&\qquad\qquad+\|\phi_{2}\|_{ L^{\infty}(0,T;H^{2}(\Omega))}
     \|\phi_{1}-\phi_{2}\|_{ L^{4}(0,T;B^{\frac{5}{2}}_{2,1}(\Omega))}
      \Big)
     \\&\leq C(R)T^{\frac{1}{4}}\|\phi_{1}-  \phi_{2}\|_{ X_{T}^{2}}.
  \end{align*}

Now using Lemma {\ref{jme}} and Lemma \ref{jme2} , one obtains
\begin{align*}
	&\|\nabla\tfrac{1}{\rho(\phi_{1})} \cdot S(\phi_{1},\mathbb{ D}\mathbf{u}_{1})-\nabla\tfrac{1}{\rho(\phi_{2})} \cdot S(\phi_{2},\mathbb{ D}\mathbf{u}_{2})\|_{ L^{2}(0,T;L^{2}(\Omega))}
	\\&= \|\nabla\tfrac{1}{\rho(\phi_{1})} \cdot S(\phi_{1},\mathbb{ D}\mathbf{u}_{1})
	-\nabla\tfrac{1}{\rho(\phi_{2})} \cdot S(\phi_{1},\mathbb{ D}\mathbf{u}_{1})
	\\&\quad+\nabla\tfrac{1}{\rho(\phi_{2})} \cdot S(\phi_{1},\mathbb{ D}\mathbf{u}_{1})
	-\nabla\tfrac{1}{\rho(\phi_{2})} \cdot S(\phi_{2},\mathbb{ D}\mathbf{u}_{2})
	\|_{ L^{2}(0,T;L^{2}(\Omega))}
	\\&\leq CT^{\frac{1}{4}}\Big(\|\nabla\tfrac{1}{\rho(\phi_{1})}-\nabla\tfrac{1}{\rho(\phi_{2})}\|_{ L^{\infty}(0,T;L^{6})}\|S(\phi_{1},\mathbb{ D}\mathbf{u}_{1})\|_{ L^{4}(0,T;L^{3}(\Omega))}
	\\&\qquad\qquad+\|\nabla\tfrac{1}{\rho(\phi_{2})}\|_{ L^{\infty}(0,T;L^{6}(\Omega))}
	\|S(\phi_{1},\mathbb{ D}\mathbf{u}_{1})-S(\phi_{2},\mathbb{ D}\mathbf{u}_{2})\|_{ L^{4}(0,T;L^{3}(\Omega))}\Big)
	\\& \leq C(R)T^{\frac{1}{4}}(\|\phi_{2}-\phi_{1}\|_{ X_{T}^{2}}+\|\mathbf u_{2}-\mathbf u_{1}\|_{ X_{T}^{1}}),
\end{align*}
and
\begin{align*}
   &\|\tfrac{1}{\rho(\phi_{1})}\nabla(\tfrac{1}{\alpha}+\phi_{1}) \Delta  {\phi}_{1}-\tfrac{1}{\rho(\phi_{2})}\nabla(\tfrac{1}{\alpha}+\phi_{2}) \Delta {\phi}_{2}\|_{ L^{2}(0,T;L^{2}(\Omega))}
  \\&= \|\tfrac{1}{\rho(\phi_{1})}\nabla \phi_{1} \Delta {\phi}_{1}
  -\tfrac{1}{\rho(\phi_{2})}\nabla \phi_{1} \Delta {\phi}_{1} +\tfrac{1}{\rho(\phi_{2})}\nabla \phi_{1} \Delta {\phi}_{1}
  -\tfrac{1}{\rho(\phi_{2})}\nabla \phi_{2} \Delta {\phi}_{1}
   \\&\quad+\tfrac{1}{\rho(\phi_{2})}\nabla \phi_{2} \Delta {\phi}_{1}
   -\tfrac{1}{\rho(\phi_{2})}\nabla \phi_{2} \Delta {\phi}_{2}
  %\\&+\frac{1}{\rho^{\varepsilon}(\phi_{2}^{\varepsilon})}\nabla(\tfrac{1}{\alpha}+\phi^{\varepsilon}_{2}) \Delta {\phi}^{\varepsilon}_{1}
%  -\frac{1}{\rho^{\varepsilon}(\phi_{2}^{\varepsilon})}\nabla(\tfrac{1}{\alpha}+\phi^{\varepsilon}_{2}) \Delta {\phi}^{\varepsilon}_{2}
\|_{ L^{2}(0,T;L^{2}(\Omega))}
  \\&\leq CT^{\frac{1}{4}}\Big(\|\tfrac{1}{\rho(\phi_{1})}-\tfrac{1}{\rho(\phi_{2})}\|_{ L^{\infty}(0,T;L^{\infty}(\Omega))}\|\nabla\phi_{1}\|_{ L^{\infty}(0,T;L^{6}(\Omega))}\|\Delta {\phi}_{1}\|_{ L^{4}(0,T;L^{3}(\Omega))}
   \\&\qquad+\|\tfrac{1}{\rho(\phi_{2})}\|_{ L^{\infty}(0,T;L^{\infty}(\Omega))}\|\nabla\phi_{1}-\nabla\phi_{2}\|_{ L^{\infty}(0,T;L^{6}(\Omega))}\|\Delta {\phi}_{1}\|_{ L^{4}(0,T;L^{3}(\Omega))}
    \\&\qquad+\|\tfrac{1}{\rho(\phi_{2})}\|_{ L^{\infty}(0,T;L^{\infty}(\Omega))}\|\nabla\phi_{2}\|_{
    L^{\infty}(0,T;L^{6}(\Omega))}\|\Delta {\phi}_{1}-\Delta {\phi}_{1}\|_{ L^{4}(0,T;L^{3}(\Omega))} \Big)
%   \\&\leq CT^{\frac{1}{4}}\big(\|\tfrac{1}{\rho(\phi_{1})}-\tfrac{1}{\rho(\phi_{2})}\|_{ L^{\infty}(0,T;L^{\infty})}\|\phi_{1}\|_{ L^{\infty}(0,T;H^{2})}\|\Delta {\phi}_{1}\|_{ L^{\infty}(0,T;L^{2})}^{\frac{1}{2}}
%   \|\Delta {\phi}_{1}\|_{ L^{2}(0,T;L^{6})}^{\frac{1}{2}}
%   \\&+\|\tfrac{1}{\rho(\phi_{2})}\|_{ L^{\infty}(0,T;L^{\infty})}\|\phi_{1}-\phi_{2}\|_{ L^{\infty}(0,T;H^{2})}\|\Delta {\phi}_{1}\|_{ L^{2}(0,T;L^{6})}^{\frac{1}{2}}\|\Delta {\phi}_{1}\|_{ L^{\infty}(0,T;L^{2})}^{\frac{1}{2}}
%   \\&+\|\tfrac{1}{\rho(\phi_{2})}\|_{ L^{\infty}(0,T;L^{\infty})}\|\phi_{2}\|_{
%   L^{\infty}(0,T;H^{2})}\|\Delta {\phi}_{1}-\Delta {\phi}_{2}\|_{ L^{2}(0,T;L^{6})}^{\frac{1}{2}}\|\Delta {\phi}_{1}-\Delta {\phi}_{2}\|_{
%   L^{\infty}(0,T;L^{2})}^{\frac{1}{2}}\big)
   \\& \leq C(R)T^{\frac{1}{4}}\|\phi_{2}-\phi_{1}\|_{ X_{T}^{2}}.
  \end{align*}
A similar argument as above yields
\begin{align*}
   &\|(\nabla\tfrac{1}{\rho(\phi_{1})} )\mathrm{div} ((\phi_{1}-\tfrac{1}{\alpha})\nabla \phi_{1}))-(\nabla\tfrac{1}{\rho(\phi_{2})} )\mathrm{div} ((\phi_{2}-\tfrac{1}{\alpha})\nabla \phi_{2})\|_{ L^{2}(0,T;L^{2}(\Omega))}
   \leq C(R)T^{\frac{1}{4}}\|\phi_{2}-\phi_{1}\|_{ X_{T}^{2}},
  \end{align*}
and
\begin{align*}
 \|\nabla(\nabla\tfrac{1}{\rho(\phi_{1})}\cdot ((\tfrac{1}{\alpha}-\phi_{1})\nabla (\phi_{1})')-\nabla(\nabla\tfrac{1}{\rho(\phi_{1})}\cdot ((\tfrac{1}{\alpha}-\phi_{2})\nabla (\phi_{2})'))\|_{ L^{2}(0,T;L^{2}(\Omega))}
   \leq C(R)T^{\frac{1}{4}}\|\phi_{2}-\phi_{1}\|_{ X_{T}^{2}}.
  \end{align*}
Analogously by Lemma {\ref{jme1}} and Lemma \ref{jme2}, there holds
\begin{align*}
&\|\nabla(\tfrac{1}{\rho(\phi_{1})} \mathrm{div}((\phi_{1}-\tfrac{1}{\alpha})\nabla \tilde{\phi})))
-\nabla(\tfrac{1}{\rho(\phi_{2})} \mathrm{div}((\phi_{2}-\tfrac{1}{\alpha})\nabla \tilde{\phi})))\|_{ L^{2}(0,T;L^{2}(\Omega))}
\\&\leq C\|\tfrac{1}{\rho(\phi_{1})} \mathrm{div} ((\phi_{1}-\tfrac{1}{\alpha})\nabla
 \tilde{\phi})-\tfrac{1}{\rho(\phi_{2})} \mathrm{div} ((\phi_{2}-\tfrac{1}{\alpha})\nabla \tilde{\phi})\|_{
 L^{2}(0,T;H^{1}(\Omega))}
\\&\leq C\|(\tfrac{1}{\rho(\phi_{1})}-\tfrac{1}{\rho(\phi_{2})})(\nabla\phi_{1}\cdot\nabla
\tilde{\phi}+(\phi_{1}-\tfrac{1}{\alpha})\Delta \tilde{\phi})\|_{ L^{2}(0,T;H^{1}(\Omega))}
\\&\quad+C\|\tfrac{1}{\rho(\phi_{2})}(\nabla(\phi_{1}-\phi_{2})\cdot\nabla
\tilde{\phi}+(\phi_{1}-\phi_{2})\Delta \tilde{\phi})\|_{ L^{2}(0,T;H^{1}(\Omega))}
\\&\leq  C(R,\tilde{\phi})T^{\frac{1}{8}}\|\phi_{2}-\phi_{1}\|_{ X_{T}^{2}},
\end{align*}
and
\begin{align*}
&\quad\|\nabla\mathrm{div} ((\tfrac{1}{\rho(\phi_{1})}-\tfrac{1}{\rho_{0}})(\tfrac{1}{\alpha}-\phi_{1})\nabla \phi'_{1}))-\nabla\mathrm{div} ((\tfrac{1}{\rho(\phi_{2})}-\tfrac{1}{\rho_{0}})(\tfrac{1}{\alpha}-\phi_{2})\nabla \phi'_{2}))\|_{ L^{2}(0,T;L^{2}(\Omega))}
\\&\quad\leq C\| ((\tfrac{1}{\rho(\phi_{1})}-\tfrac{1}{\rho_{0}})(\tfrac{1}{\alpha}-\phi_{1})\nabla \phi'_{1}))- ((\tfrac{1}{\rho(\phi_{2})}-\tfrac{1}{\rho_{0}})(\tfrac{1}{\alpha}-\phi_{2})\nabla \phi'_{2}))\|_{ L^{2}(0,T;H^{2}(\Omega))}
\\&\quad\leq C(R) T^{\frac{1}{8}} \|\phi_{2}-\phi_{1}\|_{ X_{T}^{2}},
\end{align*}
and
\begin{align*}
&\|\mathrm{div}((\tfrac{1}{\rho(\phi_{1})}-\tfrac{1}{\rho_{0}})S(\phi_{1},\mathbb{ D}\mathbf{u}_{1}))-\mathrm{div}((\tfrac{1}{\rho(\phi_{2})}-\tfrac{1}{\rho_{0}})S(\phi_{2},\mathbb{ D}\mathbf{u}_{2}))\|_{ L^{2}(0,T;L^{2}(\Omega))}
\\&\quad\leq\|((\tfrac{1}{\rho(\phi_{1})}-\tfrac{1}{\rho_{0}})S(\phi_{1},\mathbb{ D}\mathbf{u}_{1}))-((\tfrac{1}{\rho(\phi_{2})}-\tfrac{1}{\rho_{0}})S(\phi_{2},\mathbb{ D}\mathbf{u}_{2}))\|_{ L^{2}(0,T;H^{1}(\Omega))}
\\&\quad\leq  C(R)T^{\frac{1}{8}}(\|\phi_{2}-\phi_{1}\|_{ X_{T}^{2}}+\|\mathbf{u}_{1}-\mathbf{u}_{2}\|_{ X_{T}^{1}}).
\end{align*}
Therefore,
\begin{equation*}
	\|F_1(\mathbf{u}_1,\phi_1)- F_1(\mathbf{u}_2,\phi_2)\|_{L^2(Q_T)}
	\leq C(R) T^{\frac{1}{8}} \|(\mathbf{u}_1 - \mathbf{u}_2,\phi_1 - \phi_2)\|_{X_T}.
\end{equation*}

{\bf Estimation of $F_2(\mathbf{u},\phi)$.}
Recall that $ F_{2}(\mathbf{u},\phi)=-\partial_t \tilde{\phi}-\mathrm{div}(\phi  \mathbf{u}) $. Since $ \tilde{\phi} $ is fixed, we only need to proceed with the Lipschitz type for $ \mathrm{div}(\phi  \mathbf{u}) $. Namely, by the embeddings, interpolations and Lemma \ref{jme2}, one gets
\begin{align}
	\nonumber
	&\|\mathrm{div}({\phi_{1}}  \mathbf{u}_{1})-\mathrm{div}({\phi_{2}}  \mathbf{u}_{2})\|_{ L^{2}(0,T;H^{1}(\Omega))}
	\\
	\nonumber&\leq C\|\phi_{1} \mathbf{u}_{1}-\phi_{2} \mathbf{u}_{1}
	+\phi_{2}  \mathbf{u}_{1}-\phi_{2} \mathbf{u}_{2}\|_{ L^{2}(0,T;H^{2}(\Omega))}
	\\
	\nonumber&\leq C\Big(\|(\phi_{1}  -\phi_{2})  \mathbf{u}_{1}\|_{ L^{2}(0,T;H^{1}(\Omega))}
	+\|\nabla(\phi_{1}  -\phi_{2})  \mathbf{u}_{1}\|_{ L^{2}(0,T;H^{1}(\Omega))}
	\\
	\nonumber&\qquad
	+\|(\phi_{1}  -\phi_{2})  \nabla\mathbf{u}_{1}\|_{ L^{2}(0,T;H^{1}(\Omega))}+ \|\phi_{2} (\mathbf{u}_{1}-  \mathbf{u}_{2})\|_{ L^{2}(0,T;H^{1}(\Omega))}
	\\
	\nonumber&\qquad
	+
	\|\nabla\phi_{2}\cdot  (\mathbf{u}_{1}-  \mathbf{u}_{2})\|_{ L^{2}(0,T;H^{1}(\Omega))}+\|\phi_{2}  \nabla(\mathbf{u}_{1}-  \mathbf{u}_{2})\|_{ L^{2}(0,T;H^{1}(\Omega))}\Big)
	\\
	\nonumber& \leq C\Big(\|\phi_{1}-\phi_{2}\|_{ L^{\infty}(0,T;B_{2,1}^{\frac{3}{2}}(\Omega))}\|\mathbf{u}_{1}\|_{ L^{2}(0,T;H^{1}(\Omega))}
	+\|\phi_{1}-\phi_{2}\|_{ L^{4}(0,T;B_{2,1}^{\frac{5}{2}}(\Omega))}\|\mathbf{u}_{1}\|_{ L^{4}(0,T;H^{1}(\Omega))}
	\\
	\nonumber&\qquad+\|\phi_{1}-\phi_{2}\|_{ L^{\infty}(0,T;B_{2,1}^{\frac{3}{2}}(\Omega))}\|\mathbf{u}_{1}\|_{ L^{2}(0,T;H^{2})}
	+\|\phi_{2}\|_{ L^{\infty}(0,T;B_{2,1}^{\frac{3}{2}}(\Omega))}\|\mathbf{u}_{1}-\mathbf{u}_{2}\|_{ L^{2}(0,T;H^{1}(\Omega))}
	\\
	\nonumber&\qquad+\|\phi_{2}\|_{ L^{2}(0,T;B_{2,1}^{\frac{5}{2}}(\Omega))}\|\mathbf{u}_{1}-\mathbf{u}_{2}\|_{ L^{\infty}(0,T;H^{1}(\Omega))}
	\\
	\nonumber&\qquad
	+\|\phi_{2}\|_{ L^{\infty}(0,T;L^{\infty}(\Omega))}\|\mathbf{u}_{1}-\mathbf{u}_{2}\|_{ L^{2}(0,T;H^{2}(\Omega))}\Big)
	\\
	\label{div-phi-u}& \leq C(R)T^{\frac{1}{8}}(\|\mathbf{u}_{1}-\mathbf{u}_{2}\|_{ X_{T}^{1}}+\|\phi_{2}-\phi_{1}\|_{ X_{T}^{2}})+C\varepsilon_{0}\|\mathbf{u}_{1}-\mathbf{u}_{2}\|_{ X_{T}^{1}},
\end{align}
{where Remark \ref{R1} is used in the estimate in the last line.}

Then,
\begin{equation*}
	\|F_2(\mathbf{u}_1,\phi_1)- F_2(\mathbf{u}_2,\phi_2)\|_{L^{2}(0,T;H^{1}_{(0)}(\Omega))}
	\leq C(R) T^{\frac{1}{8}} \|(\mathbf{u}_1 - \mathbf{u}_2,\phi_1 - \phi_2)\|_{X_T}+C\varepsilon_{0}\|\mathbf{u}_{1}-\mathbf{u}_{2}\|_{ X_{T}^{1}}.
\end{equation*}

{\bf Estimation of $F_3(\mathbf{u},\phi)$.}
As for the boundary term, since $(\mathbf{n}\cdot (\mathrm{div}\mathbf{u}_{j})\mathbf{I})_{\tau}=0$ where $j=1,2$, we have
\begin{align*}
   &\|(\mathbf{n}\cdot S(\phi_{1},\nabla^{2} G(\mathrm{div}\mathbf{u}_{1}))_{\tau}-(\mathbf{n}\cdot S(\phi_{2},\nabla^{2} G(\mathrm{div}\mathbf{u}_{2}))_{\tau}\|_{ H^{\frac{1}{4},\frac{1}{2}}(S_{T})}
   \\&\leq\|(2\eta(\phi_{1})-2\eta(\phi_{2}))(\mathbf{n}\cdot \nabla^{2} G(\mathrm{div}\mathbf{u}_{1}))_{\tau}\|_{ H^{\frac{1}{4},\frac{1}{2}}(S_{T})} \\&\quad+\|2\eta(\phi_{2})(\mathbf{n}\cdot \nabla^{2} G(\mathrm{div}(\mathbf{u}_{1}-\mathbf{u}_{2})))_{\tau}\|_{ H^{\frac{1}{4},\frac{1}{2}}(S_{T})}
   %\\&+\|(2\eta(\phi_{1})-2\eta(\phi_{2}))(\mathbf{n}\cdot ((\nabla\cdot\mathbf{u}^{\varepsilon}_{1})\mathbf{I})_{\tau}\|_{ H^{\frac{1}{4},\frac{1}{2}}(S_{T})}+\|2\eta(\phi_{2})(\mathbf{n}\cdot (\nabla\cdot(\mathbf{u}^{\varepsilon}_{1}-\mathbf{u}^{\varepsilon}_{2})\mathbf{I})_{\tau}\|_{ H^{\frac{1}{4},\frac{1}{2}}(S_{T})}
   %\\&\leq C(R)T^{\frac{1}{4}}\|\phi_{1}-\phi_{2}\|_{X_{T}^{1}}+C(R)T^{\frac{1}{4}}\|\mathbf{u}_{1}-\mathbf{u}_{2}\|_{X_{T}^{2}}
   .
  \end{align*}

  We can note hat
  \begin{align}
	\|fg\|_{L^{2}(\partial\Omega)}&\leq C\|f\|_{L^{4}(\partial\Omega)}\|g\|_{ L^{4}(\partial\Omega)}\leq C\|f\|_{H^{\frac{1}{2}}(\partial\Omega)}\|g\|_{ H^{\frac{1}{2}}(\partial\Omega)}\label{sobolev em-01}
\end{align}
and
  \begin{align}
\|fg\|_{H^{\frac{1}{2}}(\partial\Omega)}&\leq C\|f\|_{H^{\frac{1}{2}}(\partial\Omega)}\|g\|_{B_{2,1}^{1}(\partial\Omega)}.\label{sobolev em-02}
\end{align}
 One can refer to \cite {HA1}\cite {JJ2}  for the last inequality.

It follows from {\eqref{frcsobolev},\eqref{frcsobolev-1},\eqref{HC}, \eqref{BCC}, \eqref{sobolev em-01} and \eqref{sobolev em-02}  that} 
\begin{align*}
	&\|(\eta(\phi_{1})-\eta(\phi_{2}))\mathbf{a}\|_{ H^{\frac{1}{4},\frac{1}{2}}(S_{T})}
	\\&\leq {CT^{\frac{1}{4}}}\|\eta(\phi_{1})-\eta(\phi_{2})\|_{C^{\frac{1}{2}}(0,T;H^{\frac{1}{2}}(\partial\Omega))}
	\|\mathbf{a}\|_{L^{2}(0,T;H^{\frac{1}{2}}(\partial\Omega))}
	\\&\quad +CT^{\frac{1}{4}}\|\eta(\phi_{1})-\eta(\phi_{2})\|_{C^{\frac{1}{4}}(0,T;B^{1}_{2,1}(\partial\Omega))}
	\|\mathbf{a}\|_{H^{\frac{1}{4}}(0,T;L^{2}(\partial\Omega))}
    \\&\quad {+CT^{\frac{1}{2}}\|\eta(\phi_{1})-\eta(\phi_{2})\|_{C^{\frac{1}{3}}(0,T;H^{\frac{5}{6}}(\partial\Omega))}
	\|\mathbf{a}\|_{L^{\infty}(0,T;H^{\frac{1}{2}}(\partial\Omega))}}
	\\&\quad {+CT^{\frac{1}{4}}\|\eta(\phi_{1})-\eta(\phi_{2})\|_{C^{\frac{1}{4}}(0,T;B^{1}_{2,1}(\partial\Omega))}
	\|\mathbf{a}\|_{L^{2}(0,T;H^{\frac{1}{2}}(\partial\Omega))}}
	%\\&\quad +CT^{\frac{1}{2}}\|\eta(\phi_{1})-\eta(\phi_{2})\|_{C^{\frac{1}{3}}(0,T;H^{\frac{5}{6}}(\partial\Omega))}
	%\|\mathbf{a}\|_{L^{\infty}(0,T;H^{\frac{1}{2}}(\partial\Omega))}
	\\&\leq CT^{\frac{1}{4}}\|\eta(\phi_{1})-\eta(\phi_{2})\|_{C^{\frac{1}{2}}(0,T;H^{1}(\Omega))\cap BUC(0,T;H^{2}(\Omega))}
	\big(\|\mathbf{a}\|_{ H^{\frac{1}{4},\frac{1}{2}}(S_{T})}+\|\mathbf{a}\|_{L^{\infty}(0,T;H^{\frac{1}{2}}(\partial\Omega))}\big),
\end{align*}
for $ \mathbf{a}\in H^{\frac{1}{4},\frac{1}{2}}(S_{T})\cap L^{\infty}(0,T;H^{\frac{1}{2}}(\partial\Omega)) $. { For the above estimate, we also use the following one. 
\begin{align*}
 \|(\eta(\phi_{1})-\eta(\phi_{2}))\mathbf{a}\|_{\dot{H}^{\frac{1}{4}}(0,T;L^{2}(\partial\Omega))}  
 &\leq  CT^{\frac{1}{4}}\|\eta(\phi_{1})-\eta(\phi_{2})\|_{C^{\frac{1}{4}}(0,T;B^{1}_{2,1}(\partial\Omega))}
	\|\mathbf{a}\|_{\dot{H}^{\frac{1}{4}}(0,T;L^{2}(\partial\Omega))}
    \\&\quad +C\|\eta(\phi_{1})-\eta(\phi_{2})\|_{\dot{H}^{\frac{1}{4}}(0,T;H^{\frac{1}{2}}(\partial\Omega))}
	\|\mathbf{a}\|_{L^{\infty}(0,T;H^{\frac{1}{2}}(\partial\Omega))}
 \\&\leq
    	 CT^{\frac{1}{4}}\|\eta(\phi_{1})-\eta(\phi_{2})\|_{C^{\frac{1}{4}}(0,T;B^{1}_{2,1}(\partial\Omega))}
	\|\mathbf{a}\|_{H^{\frac{1}{4}}(0,T;L^{2}(\partial\Omega))}
    \\&\quad +CT^{\frac{1}{2}}\|\eta(\phi_{1})-\eta(\phi_{2})\|_{C^{\frac{1}{3}}(0,T;H^{\frac{5}{6}}(\partial\Omega))}
	\|\mathbf{a}\|_{L^{\infty}(0,T;H^{\frac{1}{2}}(\partial\Omega))},
\end{align*}
which is obtained by using \eqref{frcsobolev-1}, \eqref{HC}, \eqref{sobolev em-01}, \eqref{sobolev em-02} and the similar idea as Lemma \ref{jme6}.}

%The concrete details can be found by mimicking the proof of Lemma \ref{jme6}.
Moreover, in light of {Lemma \ref{jme3}}, Lemma \ref{jme},  \eqref{HH} and Lemma \ref{jme5}($p=2,s=1$), one deduces
\begin{align*}
   &\|(\mathbf{n}\cdot \nabla^{2} G(\mathrm{div}\mathbf{u}))_{\tau}\|_{ H^{\frac{1}{4},\frac{1}{2}}(S_{T})}
   \\&\leq CT^{\frac{1}{2}}\|(\mathbf{n}\cdot \nabla^{2} G(\mathrm{div}\mathbf{u}))_{\tau}\|_{BUC(0,T;H^{\frac{1}{2}}(\partial\Omega))}+CT^{\frac{1}{2}}\|(\mathbf{n}\cdot \nabla^{2} G(\mathrm{div}\mathbf{u}))_{\tau}\|_{BUC(0,T;L^{2}(\partial\Omega))}
   \\&\quad+\|(\mathbf{n}\cdot \nabla^{2} G(\mathrm{div}\mathbf{u}))_{\tau}\|_{\dot{H}^{\frac{1}{4}}(0,T;L^{2}(\partial\Omega))}
   \\&\leq CT^{\frac{1}{2}}\|(\mathbf{n}\cdot \nabla^{2} G(\mathrm{div}\mathbf{u}))_{\tau}\|_{BUC(0,T;H^{1}(\Omega))}
   + CT^{\frac{1}{2}}\|(\mathbf{n}\cdot \nabla^{2} G(\mathrm{div}\mathbf{u}))_{\tau}\|_{\dot{H}^{\frac{3}{4}}(0,T;L^{2}(\partial\Omega))}
   \\&\leq CT^{\frac{1}{2}}\|\mathbf{u}\|_{X_{T}^{1}}.
  \end{align*}
 We can use again the same argument to get
\begin{align*}
   &\|( \nabla G(\mathrm{div}\mathbf{u}))_{\tau}\|_{ H^{\frac{1}{4},\frac{1}{2}}(S_{T})}
  \leq CT^{\frac{1}{2}}\|\mathbf{u}\|_{X_{T}^{1}},
  \\& \|(\mathbf{n}\cdot \nabla^{2} G(\mathrm{div}\mathbf{u}))_{\tau}\|_{L^{\infty}(0,T;H^{\frac{1}{2}}(\partial\Omega))}\leq C\|\mathbf{u}\|_{X_{T}^{1}}.
% \\& \|(\mathbf{n}\cdot (\nabla\cdot\mathbf{u}^{\varepsilon})\mathbf{I})_{\tau}\|_{ H^{\frac{1}{4},\frac{1}{2}}(S_{T})}
%  \leq CT^{\frac{1}{2}}\|\mathbf{u}^{\varepsilon}\|_{X_{T}^{2}},
%  \\&\|(\mathbf{n}\cdot (\nabla\cdot\mathbf{u}^{\varepsilon})\mathbf{I})_{\tau}\|_{L^{\infty}(0,T;H^{\frac{1}{2}}(\partial\Omega))}\leq C\|\mathbf{u}^{\varepsilon}\|_{X_{T}^{2}}
  \end{align*}
Hence, there holds
%\begin{align*}
%   &\|(\mathbf{n}\cdot S(\phi^{\varepsilon}_{1},\nabla^{2} G(\nabla\cdot\mathbf{u}^{\varepsilon}_{1}))_{\tau}-(\mathbf{n}\cdot S(\phi^{\varepsilon}_{2},\nabla^{2} G(\nabla\cdot\mathbf{u}^{\varepsilon}_{2}))_{\tau}\|_{ H^{\frac{1}{4},\frac{1}{2}}(S_{T})}
%   \\&\leq\|(2\eta(\phi_{1})-2\eta(\phi_{2}))(\mathbf{n}\cdot \nabla^{2} G(\nabla\cdot\mathbf{u}^{\varepsilon}_{1}))_{\tau}\|_{ H^{\frac{1}{4},\frac{1}{2}}(S_{T})}+\|2\eta(\phi_{2})(\mathbf{n}\cdot \nabla^{2} G(\nabla\cdot(\mathbf{u}^{\varepsilon}_{1}-\mathbf{u}^{\varepsilon}_{2}))))_{\tau}\|_{ H^{\frac{1}{4},\frac{1}{2}}(S_{T})}
%   \\&+\|(2\eta(\phi_{1})-2\eta(\phi_{2}))(\mathbf{n}\cdot (\nabla\cdot\mathbf{u}^{\varepsilon}_{1})\mathbf{I})_{\tau}\|_{ H^{\frac{1}{4},\frac{1}{2}}(S_{T})}+\|2\eta(\phi_{2})(\mathbf{n}\cdot (\nabla\cdot(\mathbf{u}^{\varepsilon}_{1}-\mathbf{u}^{\varepsilon}_{2})\mathbf{I})_{\tau}\|_{ H^{\frac{1}{4},\frac{1}{2}}(S_{T})}
%   \\&\leq C(R)T^{\frac{1}{2}}\|\phi_{1}-\phi_{2}\|_{X_{T}^{1}}+C(R)T^{\frac{1}{2}}\|\mathbf{u}^{\varepsilon}_{1}-\mathbf{u}^{\varepsilon}_{2}\|_{X_{T}^{2}}.
%  \end{align*}
\begin{align*}
   &\|2\eta(\phi_{1})(\mathbf{n}\cdot \nabla^{2} G(\mathrm{div}\mathbf{u}_{1}))_{\tau}-2\eta(\phi_{2})(\mathbf{n}\cdot \nabla^{2} G(\mathrm{div}\mathbf{u}_{1}))_{\tau}\|_{ H^{\frac{1}{4},\frac{1}{2}}(S_{T})}
   \\&\leq\|(2\eta(\phi_{1})-2\eta(\phi_{2}))(\mathbf{n}\cdot \nabla^{2} G(\mathrm{div}\mathbf{u}_{1}))_{\tau}\|_{ H^{\frac{1}{4},\frac{1}{2}}(S_{T})}
   \\&\quad+\|2\eta(\phi_{2})(\mathbf{n}\cdot \nabla^{2} G(\mathrm{div}(\mathbf{u}_{1}-\mathbf{u}_{2}))))_{\tau}\|_{ H^{\frac{1}{4},\frac{1}{2}}(S_{T})}
   \\&\leq C(R)T^{\frac{1}{4}}\|\phi_{1}-\phi_{2}\|_{X_{T}^{2}}+C(R)T^{\frac{1}{4}}\|\mathbf{u}_{1}-\mathbf{u}_{2}\|_{X_{T}^{1}},
  \end{align*}
and similarly we have
\begin{align*}
	\|(a(\phi_{1})\nabla G(\mathrm{div}\mathbf{u}_{1}))_{\tau}-(a(\phi_{2})\nabla G(\mathrm{div}\mathbf{u}_{2}))_{\tau}\|_{ H^{\frac{1}{4},\frac{1}{2}}(S_{T})}
	\leq C(R)T^{\frac{1}{4}}\|(\mathbf{u}_1 - \mathbf{u}_2,\phi_1 - \phi_2)\|_{X_T}.
\end{align*}

By combining the  above estimates  one can give
\begin{equation*}
	\|F_3(\mathbf{u}_1,\phi_1)- F_3(\mathbf{u}_2,\phi_2)\|_{H^{\frac{1}{4},\frac{1}{2}}(S_{T})}
	\leq C(R) T^{\frac{1}{4}} \|(\mathbf{u}_1 - \mathbf{u}_2,\phi_1 - \phi_2)\|_{X_T},
\end{equation*}
which completes the proof together with the control of $ F_1 $ and $ F_2 $.
\end{proof}

After these main ingredients, we now begin to provide a proof of Theorem \ref{main theorem1} with the aid of  Proposition \ref{pjme1}(the proof is defered to be given  in Section \ref{sec:proofoflinear}) and Proposition \ref{pjme}.

{\bf Proof of Theorem \ref{main theorem1}:} Let $\tilde{\phi}\in X_{T}^{2}$ be such that $\|\tilde{\phi}\|_{X_{T}^{2}}\leq 2C(\|\phi_{0}\|_{ H^{2}(\Omega)}+\|\mathbf{u}_{0}\|_{ H^{1}(\Omega)})$ for some fixed constant $C$  and
$$R_0=\max\{2C(\|\phi_{0}\|_{ H^{2}(\Omega)}+\|\mathbf{u}_{0}\|_{ H^{1}(\Omega)}),3\|\mathbb{T}\big(L(\tilde{\phi})\big)^{-1}F(\tilde{\mathbf{v}})\|_{ X_{T}}\},$$
where $\tilde{\mathbf{v}}=(0,\tilde{\phi})$, $\mathbb{T}:X_{T}\rightarrow X_{T}$ is defined by
\begin{equation}
\mathbb{T}\binom{\mathbf{u}}{\phi'}=\binom{\mathbf{u}}{\phi'+\widetilde{\phi}}.\label{def T}
\end{equation}

 We will construct a solution in $B_{R_0}(0)\subset X_{T}$ for sufficiently
small $0 < T \ll 1$ with the aid of the Banach contraction mapping principle. Firstly, we will prove  the invertibility  of operator $L(\phi)$ when  $\phi\in\tilde{X}_{T}^{2}$ which is defined by
$$\tilde{X}_{T}^{2}= BUC(0,T;H^{2}(\Omega))\cap H^{1}(0,T;H^{1}(\Omega))\cap H^{\frac{1}{2}}(0,T;H^{2}(\Omega)).$$

Next, let $\phi_{1},\phi_{2}\in B_{R_0}(0)\subset \tilde{X}_{T}^{2}$ with $\phi_{1}|_{t=0}=\phi_{2}|_{t=0}$. In order to prove that the operator $L(\phi)$ is invertible, we firstly observe that
\begin{align*}
&\|(\mathbf{n}\cdot S(\phi_{1},\mathbb{ D}  P_{\sigma}\mathbf{u}))_{\tau}-(\mathbf{n}\cdot S(\phi_{2},\mathbb{ D}  P_{\sigma}\mathbf{u}))_{\tau}\|_{H^{\frac{1}{4},\frac{1}{2}}(S(T))}
\\&\leq CT^{\frac{1}{8}}\|\phi_{1}-\phi_{2}\|_{\tilde{X}_{T}^{2}}\| \mathbf{u}\|_{ X_{T}^{1}}.
\end{align*}
 In fact,
\begin{align*}
&\|(\mathbf{n}\cdot S(\phi_{1},\mathbb{ D}  P_{\sigma}\mathbf{u}))_{\tau}-(\mathbf{n}\cdot S(\phi_{2},\mathbb{ D}  P_{\sigma}\mathbf{u}))_{\tau}\|_{H^{\frac{1}{4},\frac{1}{2}}(S(T))}
\\&\leq C\|(\mathbf{n}\cdot S(\phi_{1},\mathbb{ D}  P_{\sigma}\mathbf{u}))_{\tau}-(\mathbf{n}\cdot S(\phi_{2},\mathbb{ D}  P_{\sigma}\mathbf{u}))_{\tau}\|_{H^{\frac{1}{2},1}(Q(T))}
\\&\leq CT^{\frac{1}{8}}\|\phi_{1}-\phi_{2}\|_{C^{\frac{1}{8}}(0,T;B^{1}_{3,1}(\partial\Omega))}
\|(\mathbf{n}\cdot \mathbb{ D}  P_{\sigma}\mathbf{u})_{\tau}\|_{L^{2}(0,T;H^{1}(\Omega))}
\\&\quad+CT^{\frac{1}{6}}\|\phi_{1}-\phi_{2}\|_{H^{\frac{2}{3}}(0,T;L^{\infty}(\Omega))}
\|(\mathbf{n}\cdot \mathbb{ D}  P_{\sigma}\mathbf{u})_{\tau}\|_{L^{\infty}(0,T;L^{2}(\Omega))}
\\&\quad+CT^{\frac{1}{8}}\|\phi_{1}-\phi_{2}\|_{C^{\frac{1}{8}}(0,T;L^{\infty}(\Omega))}
\|(\mathbf{n}\cdot \mathbb{ D}  P_{\sigma}\mathbf{u})_{\tau}\|_{H^{\frac{1}{2}}(0,T;L^{2}(\Omega))}
%\\&\quad+CT^{\frac{1}{8}}\|\phi_{1}-\phi_{2}\|_{C^{\frac{1}{8}}(0,T;L^{\infty}(\Omega))}
%\|(\mathbf{n}\cdot \mathbb{ D}  P_{\sigma}\mathbf{u})_{\tau}\|_{H^{\frac{1}{2}}(0,T;L^{2}(\Omega))}
\\&\leq CT^{\frac{1}{8}}\|\phi_{1}-\phi_{2}\|_{\tilde{X}_{T}^{2}}\| \mathbf{u}\|_{ X_{T}^{1}} ,
\end{align*}
where we have employed Lemma \ref{jme6}.

 Thanks to Lemma {\ref{jme2}}, we then obtain
\begin{align}
   &\|L(\phi_{1})\mathbf{v}-L(\phi_{2})\mathbf{v}\|_{ Y_{T}}
  \nonumber\\&\leq C\big(\|\frac{1}{\rho(\phi_{0})}S(\phi_{1},\mathbb{ D}\mathbf{u}) -\frac{1}{\rho(\phi_{0})}S(\phi_{2},\mathbb{ D}\mathbf{u})\|_{ L^{2}(0,T;H^{1}(\Omega))}\nonumber
  \\&\quad+\|\frac{1}{\rho(\phi_{0})}(\tfrac{1}{\alpha}-\phi_{1})\nabla \phi'-\frac{1}{\rho(\phi_{0})}(\tfrac{1}{\alpha}-\phi_{2})\nabla \phi'\|_{ L^{2}(0,T;H^{2}(\Omega))} \nonumber\\&\quad
  %+\|\phi_{1}  \mathbf{u}-\phi_{2}  \mathbf{u}\|_{ L^{2}(0,T;H^{2}(\Omega))}
  +\|(\mathbf{n}\cdot S(\phi_{1},\mathbb{ D}  P_{\sigma}\mathbf{u}))_{\tau}-(\mathbf{n}\cdot S(\phi_{2},\mathbb{ D}  P_{\sigma}\mathbf{u}))_{\tau}\big|_{\partial\Omega}\|_{H^{\frac{1}{4},\frac{1}{2}}(S(T))}\big)\nonumber
  \\&\quad+\|(a(\phi_{1}) P_{\sigma}\mathbf{u})_{\tau}-(a(\phi_{2}) P_{\sigma}\mathbf{u})_{\tau}\|_{ H^{\frac{1}{4},\frac{1}{2}}(S_{T})}
  \nonumber\\& \leq C\|\frac{1}{\rho(\phi_{0})}(S(\phi_{1},\mathbb{ D}\mathbf{u})-S(\phi_{2},\mathbb{ D}\mathbf{u}))\|_{ L^{2}(0,T;H^{1}(\Omega))}
  +\|\frac{1}{\rho(\phi_{0})}(\phi_{1}-\phi_{2})\nabla \phi'\|_{ L^{2}(0,T;H^{2}(\Omega))}
  %+\|\frac{1}{\rho^{\varepsilon}(\phi^{\varepsilon}_{2})}\phi^{\varepsilon}_{2}\nabla (\phi'^{\varepsilon}_{1}-\phi'^{\varepsilon}_{2})\|_{ L^{\infty}(0,T;H^{2}(\Omega))}
 %\\& +\|\frac{1}{\rho^{\varepsilon}(\phi^{\varepsilon}_{2})}\|_{ L^{\infty}(0,T;H^{2})}
%  \|(\phi^{\varepsilon}_{1}-\phi^{\varepsilon}_{2})\|_{ L^{\infty}(0,T;H^{2})}\|\nabla \phi'^{\varepsilon}_{1}\|_{ L^{2}(0,T;H^{1})}
 % +\|(\phi^{\varepsilon}_{1}-\phi^{\varepsilon}_{2})\nabla \phi'^{\varepsilon}_{1}\|_{ L^{2}(0,T;B_{3,1}^{1})}\|(\phi^{\varepsilon}_{1}-\phi^{\varepsilon}_{2})\nabla \phi'^{\varepsilon}_{1}\|_{ L^{2}(0,T;B_{3,1}^{1})})
 %\\&\quad+C\|\phi_{1} -\phi_{2}\|_{ L^{\infty}(0,T;B_{3,1}^{1}(\Omega))}\| \mathbf{u}\|_{ L^{2}(0,T;H^{2}(\Omega))}
%+C\|\phi_{1} -\phi_{2}\|_{ L^{\infty}(0,T;H^{2})}\| \mathbf{u}\|_{ L^{2}(0,T;B_{3,1}^{1}(\Omega))}%+\| (\phi^{\varepsilon}_{1}\nabla G(\nabla\cdot\mathbf{u}^{\varepsilon})-(\phi^{\varepsilon}_{2},\nabla G(\nabla\cdot\mathbf{u}^{\varepsilon})\|_{ H^{\frac{1}{4},\frac{1}{2}}(S_{T})}
\nonumber\\&\quad+CT^{\frac{1}{8}}\|\phi_{2}-\phi_{1}\|_{\tilde X_{T}^{2}}\| \mathbf{u}\|_{ X_{T}^{1}}
   \nonumber\\& \leq C(R_0)T^{\frac{1}{8}}\|\phi_{2}-\phi_{1}\|_{\tilde X_{T}^{2}}\| \mathbf{v}\|_{ X_{T}}.\label{Lip con}
  \end{align}

Noting that $\phi_{0}\in \tilde{X}_{T}^{2}$,
hence there is some $0 < T_{0} \leq 1$ such that
\begin{align*}
&\|L(\phi)\mathbf{v}-L(\phi_{0})\mathbf{v}\|_{ Y_{T}}\leq\frac{1}{4C}\|\mathbf{v}\|_{ X_{T}}, \quad\text{for all $0<T<T_{0}$, $\|\phi\|_{ X_{T}^{2}}\leq R_0$ },\quad
\\&\|L(\phi_{1})\mathbf{v}-L(\phi_{2})\mathbf{v}\|_{ Y_{T}}\leq\frac{1}{4C}\|\mathbf{v}\|_{ X_{T}},
\quad\text{for all $0<T<T_{0}$, $\|\phi_{j}\|_{ X_{T}^{2}}\leq R_0,j=1,2$ ,}\quad
\end{align*}
where $C=C(T_{0})|_{T_{0}=1}$ in Proposition \ref{pjme1}. This implies that $L(\phi):X_{T}\rightarrow Y_{T}$ is invertible and $\|\big(L(\phi)\big)^{-1}\|_{\mathcal L(X_{T}, Y_{T})}\leq2C$   because of $\|L(\phi_{0})\mathbf{v}\|_{ Y_{T}}\geq\frac{ 1}{C}\|\mathbf{v}\|_{ X_{T}}$ which comes from {Proposition \ref{pjme1}}. Furthermore, due to \eqref{Lip con} one has
\begin{align*}
&\|\big(L(\phi_{1})\big)^{-1}-\big(L(\phi_{2})\big)^{-1}\|_{ \mathcal L(Y_{T}, X_{T})}
\leq 8C^{2}\|L(\phi_{1})-L(\phi_{2})\|_{ \mathcal L(X_{T}, Y_{T})}\leq C(R_0)T^{\frac{1}{8}}\|\phi_{2}-\phi_{1}\|_{ X_{T}^{2}}.
\end{align*}
which further imply that
\begin{align*}
\|\mathbb{T}\big(L(\phi_{1})\big)^{-1}-\mathbb{T}\big(L(\phi_{2})\big)^{-1}\|_{ \mathcal L(Y_{T},X_{T})}
\leq C(R_0)T^{\frac{1}{8}}\|\phi_{1}-\phi_{2}\|_{X_{T}^{2}}.
\end{align*}

Employing the invertibility of operator $L(\phi)$, we  transform the system (\ref{model1-9}) to (\ref{BCs333}) into the following fixed point problem
\begin{align}
\mathbf{v}=\mathbb{T}\big(L(\phi)\big)^{-1}F(\mathbf{v}),\label{fix problem}
\end{align}
 where \begin{equation*}
\mathbf{v}=(\mathbf{u},\phi).
\end{equation*}
Then using Proposition {\ref{pjme}}, we obtain
\begin{align}
\nonumber
&\|\mathbb{T}\big(L(\phi_{1})\big)^{-1}F(\mathbf{v}_{1})
-\mathbb{T}\big(L(\phi_{2})\big)^{-1}F(\mathbf{v}_{2})\|_{ X_{T}}
\\
\nonumber &\leq\|\mathbb{T}\big(L(\phi_{1})\big)^{-1}-\mathbb{T}\big(L(\phi_{2})\big)^{-1}\|_{ \mathcal L(Y_{T},X_{T})}\|F(\mathbf{v}_{1})\|_{ Y_{T}}
\\
\nonumber &\quad+\|\big(L(\phi_{2})\big)^{-1}\|_{\mathcal L(Y_{T}, X_{T})}
\|(F(\mathbf{v}_{1})-F(\mathbf{v}_{2}))\|_{ Y_{T}}
\\
\nonumber &\leq C(R_0)T^{\frac{1}{8}}\|\phi_{1}-\phi_{2}\|_{X_{T}}+ (C(R_0,T)+ C\varepsilon_{0})\|\mathbf{v}_{1}-\mathbf{v}_{2}\|_{X_{T}}\\&\leq \frac{1}{4}\|(\mathbf{v}_{1}-\mathbf{v}_{2},\phi_1-\phi_2)\|_{X_{T}},
\end{align}
when $T$ and $\varepsilon_{0}$  are taken sufficient small for all $\mathbf{v}_{j}=(\mathbf{u}_{j},\phi_{j})\in X_{T} $ with
$\mathbf{v}_{j}|_{t=0}=(\mathbf{u}_{0},\phi_{0})$ and $\|\mathbf{v}_{j}\|_{X_{T}}\leq R_0$. Moreover, from $\|\mathbf{v}\|_{X_{T}}\leq R_0$, $\mathbf{v}|_{t=0}=(\mathbf{u}_{0},\phi_{0})$ and above arguments, we have
\begin{align*}
&\|\mathbb{T}\big(L(\phi)\big)^{-1}F(\mathbf{v})\|_{X_{T}}
\\&\leq\|\mathbb{T}\big(L(\phi)\big)^{-1}F(\mathbf{v})-\mathbb{T}\big(L(\tilde{\phi})\big)^{-1}F(\tilde{\mathbf{v}})\|_{X_{T}}+
\|\mathbb{T}\big(L(\tilde{\phi})\big)^{-1}F(\tilde{\mathbf{v}})\|_{X_{T}}
\\&\leq\frac{1}{4}\cdot \frac{2}{3} R_0+\frac{1}{3} R_0<R_0.
\end{align*}
It follows from  the contraction mapping principle that there is a unique solution $\mathbf{v}=(\mathbf{u},\phi)$ with $\|\mathbf{v}\|_{X_{T}}\leq R_0$ of \eqref{fix problem}. Therefore there is a solution $\mathbf{v}=(\mathbf{u},\phi)\in X_{T}$ of (\ref{model1-9}) to (\ref{BCs333}) with $\|\mathbf{v}\|_{X_{T}}\leq R_0$ and $(\mathbf{u}_{0},\phi_{0})\in (H_{n}^{1}(\Omega), H^{2}_{N}(\Omega))$.

The uniqueness of the solution can be shown by the following argument. We assume that there is another solution $(\hat{\mathbf{u}},\hat{\phi})\in X_{T}$ and $R'=\max(R_0,\|(\hat{\mathbf{u}},\hat{\phi})\|_{X_{T}})$. By the above fixed point argument, we can derive a unique solution $(\check{\mathbf{u}},\check{\phi})\in X_{T'}$ to \eqref{model31} with $\|(\check{\mathbf{u}},\check\phi)\|_{X_{T'}}\leq R'$ where $0<T'\leq T$. Thus $(\mathbf{u},\phi)|_{(0,T')}=(\hat{\mathbf{u}},\hat{\phi})_{(0,T')}= (\check{\mathbf{u}},\check\phi)_{(0,T')}$. With the temporal trace embedding, cf.\cite{MR}, one knows the $ (\mathbf u, \phi)|_{t=T'} \in H_n^1(\Omega) \times H_N^2(\Omega) $. Performing this
argument finitely many times (with a shift in time), we conclude that $(\mathbf{u},\phi)|_{(0,T)}=(\hat{\mathbf{u}},\hat{\phi})_{(0,T)}$.

%\begin{Remark}
%We also need to give a  description of that R is a constant only related with $\|\phi^{\varepsilon}_{0}\|_{H^{2}(\Omega)}+\|\mathbf{u}^{\varepsilon}_0\|_{H^{1}(\Omega)}$ %which is very important when we consider asymptotic behavior as $\varepsilon\rightarrow0$ in the last section
%. We deduce this from
%\begin{align*}
%&\|SL^{-1}(\tilde{\phi^{\varepsilon}})F(\tilde{\mathbf{v}^{\varepsilon}})\|_{ X_{T}}
%\\&\leq \|\tilde{\phi^{\varepsilon}}\|_{X_{T}^{2}(\Omega)}+\|L^{-1}(\tilde{\phi^{\varepsilon}})F(\tilde{\mathbf{v}^{\varepsilon}})\|_{X_{T}}
%\\&\leq C(\|\phi^{\varepsilon}_{0}\|_{H^{2}(\Omega)}+\|\mathbf{u}^{\varepsilon}_0\|_{H^{1}(\Omega)}) ,
%\end{align*}
%for some constant $C>0$ and %$\|\mathbf{u}^{\varepsilon}\|_{X_{T}^{1}}\leq\frac{1}{2}C_{1}(\|\phi^{\varepsilon}_{0}\|_{H^{2}(\Omega)}+\|\mathbf{u}^{\varepsilon}_0\|_{H^{1}(\Omega)})$
% , hence%, when $\alpha$ is fixed and sufficient small
% , we get $R\leq C(\|\phi^{\varepsilon}_{0}\|_{H^{2}(\Omega)}+\|\mathbf{u}^{\varepsilon}_0\|_{H^{1}(\Omega)})$.
%\end{Remark}

\section{Proof of Proposition \ref{pjme1}}
\label{sec:proofoflinear}

In this section, we give the proof of Proposition \ref{pjme1} by using the method in \cite{HA1}. Namely, we are devoted to showing the well-posdenss of the corresponding linear system \begin{subequations}
 \label{model-linear-phi0}
\begin{align}
\label{model-linear-phi0-1}
&\partial_t  \mathbf{u} -\mathrm{div}(\tfrac{1}{\rho_{0}}S(\phi_{0},\mathbb{ D}\mathbf{u}))+ \nabla\mathrm{div} (\tfrac{1}{\rho_{0}}(\tfrac{1}{\alpha}-\phi_{0})\nabla \phi')=\mathbf{f}_{1},\quad \text{in } Q_T,\\
\label{model-linear-phi0-2}
&\partial_t \phi'-\tfrac{1}{\alpha} \mathrm{div} \mathbf{u}=f_{2},\quad \text{in } Q_T,
\\ \label{model-linear-phi0-3}
&\mathbf{u} \cdot \mathbf{n}=0, \quad \text{on } \partial \Omega \times (0,T), \\&
\label{model-linear-phi0-4}
\nabla \phi'\cdot \mathbf{n}=\nabla \mu_p \cdot \mathbf{n} =0, \quad \text{on } \partial \Omega \times (0,T),
\\&\label{model-linear-phi0-5}
(\mathbf{n}\cdot2\eta(\phi_{0}) \mathbb{ D}( \mathbf{w}))_{\tau}+(a(\phi_{0})\mathbf{w})_{\tau}=\mathbf{a}, \quad \text{on } \partial \Omega \times (0,T), \\&
\label{model-linear-phi0-6}
(\mathbf{u}, \phi')|_{t=0}=( \mathbf{u}_0, \phi'_{0}), \quad \text{in } \Omega,
\end{align}
where $(\mathbf{f}_{1},f_{2},\mathbf{a},\mathbf{u}_0, \phi'_{0})\in Y_{T}$, $(\mathbf{u}, \phi')\in X_{T}$ and $\mathbf{w}=P_{\sigma}\mathbf{u}$.
\end{subequations}

\subsection{The invertibility of  $L(\phi_{0})$}

Without loss of generality we can assume  $\mathbf{a}=0$ {by Lemma \ref{jme4}}.
Now we reformulate \eqref{model-linear-phi0} above in an appropriate way assuming that $(\mathbf{u}, \phi')\in X_{T}$. Inspired by \cite{HA1}, we would decompose $\mathbf{u}$ into two parts by Helmholtz decomposition. More precisely, applying $P_{\sigma}$ and $I-P_{\sigma}$ to \eqref{model-linear-phi0-1}, we obtain equivalent system of \eqref{model-linear-phi0-1} as
\begin{subequations}
  \label{model12}
  \begin{align}
    \label{model1-92}
    &\partial_t  \mathbf{w} -P_{\sigma}\mathrm{div}(\tfrac{1}{\rho_{0}}S(\phi_{0},\mathbb{ D}\mathbf{w})) -P_{\sigma}\mathrm{div}(\tfrac{1}{\rho_{0}}S(\phi_{0},\mathbb{ D}\nabla G(g)))=P_{\sigma}\mathbf{f}_{1},\\
    \label{model1-102}
    &\partial_t  \nabla G(g) -(I-P_{\sigma})\mathrm{div}(\tfrac{1}{\rho_{0}}S(\phi_{0},\mathbb{ D}\mathbf{u}))+ \nabla\mathrm{div} (\tfrac{1}{\rho_{0}}(\tfrac{1}{\alpha}-\phi_{0})\nabla \phi'))=(I-P_{\sigma})\mathbf{f}_{1},
  \end{align}
\end{subequations}
where $\mathbf{w}=P_{\sigma}\mathbf{u}$, $g=\mathrm{div}\mathbf{u}$ and $\mathbf{u}=\mathbf{w}+\nabla G(g)$.

Note that
  \begin{align*}
    & P_{\sigma}\mathrm{div}(\tfrac{1}{\rho_{0}}S(\phi_{0},\mathbb{ D}\nabla G(g)))\\&=P_{\sigma}\mathrm{div}(\tfrac{1}{\rho_{0}}S(\phi_{0},\nabla^{2} G(g)))
    \\&=P_{\sigma}(\nabla\nu(\phi_{0})\cdot\nabla^{2} G(g))+P_{\sigma}(\nu(\phi_{0})\nabla g)
    +P_{\sigma}(\nabla(\gamma(\phi_{0})g))
     \\&=P_{\sigma}(\nabla\nu(\phi_{0})\cdot\nabla^{2} G(g))+P_{\sigma}(\nabla(\nu(\phi_{0}) g))-P_{\sigma}(\nabla\nu(\phi_{0}) g)
     +P_{\sigma}(\nabla(\gamma(\phi_{0})g))
    \\&=P_{\sigma}(\nabla\nu(\phi_{0})\cdot \nabla^{2} G(g))-P_{\sigma}(\nabla\nu(\phi_{0}) g)\triangleq B_{1}g,
  \end{align*}
where $\nu(\phi_{0})=\frac{2\eta(\phi_{0})}{\rho_{0}}$ and $\gamma(\phi_{0})=-\frac{2}{3}\frac{\eta(\phi_{0})}{\rho_{0}}$.

Moreover, taking $\nabla\varphi$  as the test function in the weak formulation of (\ref{model1-102}), we have
\begin{align*}
    & -\langle\partial_{t}g,\varphi\rangle_{L^{2}(0,T;H_{(0)}^{-1}(\Omega))}-(\mathrm{div}(\tfrac{1}{\rho_{0}}S(\phi_{0},\mathbb{ D}\mathbf{u})),\nabla\varphi)_{L^{2}(0,T;L^{2}(\Omega))}
    \\&\quad+(\nabla\mathrm{div} (\tfrac{1}{\rho_{0}}(\tfrac{1}{\alpha}-\phi_{0})\nabla \phi')),\nabla\varphi)_{L^{2}(0,T;L^{2}(\Omega))}\\&=((I-P_{\sigma})\mathbf{f}_{1},\nabla\varphi)_{L^{2}(0,T;L^{2}(\Omega))},
  \end{align*}
for all $\varphi\in C_{0}^{\infty}(0,T;H_{(0)}^{1}(\Omega))$.

Direct computation leads to
\begin{align*}
     \mathrm{div}(\tfrac{1}{\rho_{0}}S(\phi_{0},\nabla^{2} G(g))&=\nabla\nu(\phi_{0})\cdot\nabla^{2} G(g)+\nu(\phi_{0})\nabla g+\nabla(\gamma(\phi_{0})g)
    \\&=\nabla\nu(\phi_{0})\cdot\nabla^{2} G(g)+\nabla(b(\phi_{0}) g)-\nabla\nu(\phi_{0}) g,
  \end{align*}
where $b(\phi_{0})=\nu(\phi_{0})+\gamma(\phi_{0})$. Hence we obtain
\begin{align*}
    & (\mathrm{div}(\tfrac{1}{\rho_{0}}S(\phi_{0},\mathbb{ D}\nabla G(g)),\nabla\varphi) \\&=-(\Delta_{N}(b(\phi_{0}) g),\varphi)_{H_{(0)}^{-1}(\Omega),H_{(0)}^{1}(\Omega)}
  \\&\quad-(\mathrm{div}_{n}(\nabla\nu(\phi_{0})\cdot\nabla^{2} G(g)),\varphi)_{H_{(0)}^{-1},H_{(0)}^{1}}  
  +(\mathrm{div}_{n}(\nabla\nu(\phi_{0}) g),\varphi)_{H_{(0)}^{-1}(\Omega),H_{(0)}^{1}(\Omega)}
  \\&\triangleq-(\Delta_{N}(b(\phi_{0}) g),\varphi)_{H_{(0)}^{-1}(\Omega),H_{(0)}^{1}(\Omega)}+(B_{2}g,\varphi)_{H_{(0)}^{-1}(\Omega),H_{(0)}^{1}(\Omega)}.
  \end{align*}

We then reformulate \eqref{model-linear-phi0-2}-(\ref{model1-102}) as
\begin{align*}
\partial_{t}\mathbf{v}+\mathbf{A}\mathbf{v}+\mathbf{B}\mathbf{v}=
\left(\begin{array}{c}
f_{2} \\
\mathrm{div}_{n}(I-P_{\sigma})\mathbf{f}_{1}
\\
P_{\sigma}\mathbf{f}_{1}
\end{array}\right)\triangleq\mathbf{f},
\end{align*}
with
\begin{align*}
\mathbf{v}|_{t=0}=
(\phi'_{0},g_{0},\mathbf{w}_{0})^T,
\end{align*}
where $\mathbf{v}=(\phi',g,\mathbf{w})^{T}$, $g_{0}=\mathrm{div} \mathbf{u}_{0}$, $\mathbf{u}_{0}=\mathbf{w}_{0}+\nabla G(g_{0})$. Moreover, the operators above are defined through
\begin{align*}
\mathbf{A}\mathbf{v}=\begin{pmatrix}A_{1} &A_{2} \\0 & -{P_{\sigma}\mathrm{div}(\tfrac{1}{\rho_{0}}S(\phi_{0},\mathbb{ D}}))\end{pmatrix}
\left(\begin{array}{c}
(\phi',g)^{T} \\
\mathbf{w}
\end{array}\right),
\end{align*}
where we have
\begin{align*}
A_{1}=\begin{pmatrix}0 &-\tfrac{1}{\alpha}P_{0} \\ \Delta_{N}\mathrm{div} (\tfrac{1}{\rho_{0}}(\tfrac{1}{\alpha}-\phi_{0})\nabla\cdot) & -\Delta_{N}(b(\phi_{0}) \cdot)\end{pmatrix},
\quad
A_{2}\mathbf{w}=
\left(\begin{array}{c}
0 \\
-\mathrm{div}_{n}\mathrm{div}(\tfrac{1}{\rho_{0}}S(\phi_{0},\mathbb{ D}))
\end{array}\right),
\end{align*}
and
\begin{align*}
\mathbf{B}\mathbf{v}=(0,B_{2}g,-B_{1}g)^T.
\end{align*}
The domains of $\mathbf{A},\mathbf{B}$ are
\begin{align*}
D(\mathbf{A})=D(\mathbf{B})=\{(\phi',g, \mathbf{w})^{T}:(\phi',g)^{T}\in D(A_{1}), \mathbf{w} \in H^{2}(\Omega)\cap L_{\sigma}^{2}(\Omega):T_{a}\mathbf{w}=0\},
\end{align*}
with
\begin{align*}
	D(A_{1})= (H^{3}(\Omega)\cap H_{N}^{2}(\Omega))\times H_{(0)}^{1}(\Omega),
\end{align*}
where
$$T_{a}\mathbf{w}=(\mathbf{n}\cdot S(\phi_{0},\mathbb{ D}\mathbf{w}))_{\tau}+(a(\phi_{0})\mathbf{w})_{\tau}\big|_{\partial\Omega}.$$
It is noticed that $\mathbf{A},\mathbf{B}$ are unbounded operators on $H$
\begin{align*}
&H=H_{1}\times L_{\sigma}^{2}(\Omega),
\end{align*}
where
$H_{1}\coloneqq H^{1}(\Omega)\times H^{-1}_{(0)}(\Omega)$.

In the following we give the boundedness of $B_i$, $i = 1,2$ operators.
\begin{Lemma}\label{lemma4.1}
Given $\phi_{0}\in H^{2}(\Omega)$, then there exist constant  $C(\phi_{0})$ such that for all $g\in H^{1}_{(0)}(\Omega)$, there hold
\begin{align}
&\label{lemma4.11}\|B_{1}(g)\|_{L^{2}(\Omega)}\leq C(\phi_{0})\|g\|_{H^{\frac{1}{2}}(\Omega)},
\\&\label{lemma4.12} \|B_{2}(g)\|_{H^{-1}_{(0)}(\Omega)}\leq C(\phi_{0})\|g\|_{H^{\frac{1}{2}}(\Omega)}.
\end{align}

\end{Lemma}
\begin{proof}
For (\ref{lemma4.11}), by the property $\Delta_{N}^{-1}:H^{s}\rightarrow H^{s+2}(0\leq s\leq1)$ and the Sobolev embedding $H^{2}\hookrightarrow L^{\infty}$, we have
\begin{align*}
\|B_{1}g\|_{L^{2}_{\sigma}(\Omega)}&=\|P_{\sigma}(\nabla\nu(\phi_{0})\cdot \nabla^{2} G(g))+P_{\sigma}(\nabla\nu(\phi_{0}) g)\|_{L^{2}(\Omega)}
\\& \leq C\|\nabla\nu(\phi_{0})\|_{L^{6}(\Omega)}\|\nabla^{2} G(g))\|_{L^{3}(\Omega)}+C\|\nabla\nu(\phi_{0})\|_{L^{6}(\Omega)}\| g\|_{L^{3}(\Omega)}
\\& \leq C(\phi_{0})\|g\|_{H^{\frac{1}{2}}(\Omega)}.
\end{align*}
 (\ref{lemma4.12}) can be obtained by similar arguments, in fact, we have
\begin{align*}
 \|B_{2}g\|_{H^{-1}_{(0)}(\Omega)}&\leq \|\nabla\nu(\phi_{0})\cdot\nabla^{2} G(g)+\nabla\nu(\phi_{0}) g\|_{L^{2}(\Omega)}
\\& \leq \|\nabla\nu(\phi_{0})\|_{L^{6}(\Omega)}\|\nabla^{2} G(g)\|_{L^{3}(\Omega)}+\|\nabla\nu(\phi_{0}) \|_{L^{6}(\Omega)}\| g\|_{L^{3}(\Omega)}
\\& \leq  C(\phi_{0})\|g\|_{H^{\frac{1}{2}}(\Omega)}.
\end{align*}
\end{proof}

Next we use the triangle structure of $\mathbf{A}$ to prove  that $-\mathbf{A}$ generate an analytic semigroup if $-A_{1}$ and $P_{\sigma}\mathrm{div}(\tfrac{1}{\rho_{0}}S(\phi_{0},\cdot))$ generate analytic semigroups.

{The following are three important conclusions which follows similarly as in \cite{HA1}. We record the result below and omit the proof.}
\begin{Lemma}\label{lemma4.2}
$-A_{1}$ produces a bounded analytic semigroup. In addition, the norm {$\|A_{1}(\phi',g)^{T}\|_{L^{2}(\Omega)}+|m(\phi')|$} and $\|\phi'\|_{H^{3}(\Omega)}+\|g\|_{H^{1}(\Omega)}$ are equivalent.
\end{Lemma}

%\begin{Proposition}
% Let $(u,v)$ be a bounded solution of (\ref{linear equation of error}), then there exist $\varepsilon_0>0, \eta_0>0$ such that for any $\varepsilon\in (0,\varepsilon_0), \eta\in(0,\eta_0)$, there holds
%\begin{align}\label{e}
%&\int_0^1\int_0^{2\pi}r(u_\theta^2+v_\theta^2)d\theta dr+\varepsilon^2\Big(\int_{0}^{1}\int_0^{2\pi}\frac{ (u_\theta+v)^2+(v_\theta-u)^2}{r}d\theta dr+\int_{0}^{1}\int_0^{2\pi}r\big(u_r^2+v_r^2\big)d\theta dr\Big)\nonumber\\
%&\leq C\varepsilon^{-2}\int_{0}^{1}\int_0^{2\pi}\Big(\frac{F^2_u}{r}+\frac{F^2_v}{r} \Big)d\theta dr.
%\end{align}
%\end{Proposition}
\begin{Proposition}\label{Proposition4.3}
 $-\mathbf{A}$ generates a bounded analytic semigroup on $H$ and $\|\mathbf{A}\mathbf{v}\|_{H}+\|\mathbf{v}\|_{H}$ is equivalent to $\|\phi'\|_{H^{3}(\Omega)}+\|g\|_{H^{1}(\Omega)}+\|\mathbf{w}\|_{H^{2}(\Omega)}$.
\end{Proposition}

\begin{Proposition}\label{Proposition4.4}
Given $\phi_{0}\in H^{2}(\Omega)$, then $-(\mathbf{A}+\mathbf{B})$ generates an analytic semigroup. Moreover we deduce that $\|(\mathbf{A}+\mathbf{B})\mathbf{v}\|_{H}+\|\mathbf{v}\|_{H}$
 and $\|\phi'\|_{H^{3}(\Omega)}+\|g\|_{H^{1}(\Omega)}+\|\mathbf{w}\|_{H^{2}(\Omega)}$ are equivalent.
\end{Proposition}

%\begin{Lemma}\label{lemma4.5} \cite{HA1}
%Denote $\mathbf{A}$, $D(\mathbf{A})$ and $H$ as before, one gets that:
%\begin{align*}
%(D(\mathbf{A}),H)_{\frac{1}{2},2}=H_{N}^{2}(\Omega)\times L^{2}_{(0)}(\Omega)\times( H^{1}(\Omega)\cap L_{\sigma}(\Omega)).
%\end{align*}
%\end{Lemma}
%\begin{proof}
% (\ref{lemma4.12}) can be obtained by similar arguments.  It can be stated as
%\begin{align*}
%& \|B_{2}(g)\|_{H^{-1}_{(0)}(\Omega)}\leq \|\nabla\nu(\phi^{\varepsilon}_{0})\cdot\mathbb{ D}\nabla G(g)+\nabla\nu(\phi^{\varepsilon}_{0}) g\|_{L^{2}(\Omega)}
%\\& \leq \|\nabla\nu(\phi^{\varepsilon}_{0})\|_{L^{6}(\Omega)}\|\mathbb{ D}\nabla G(g)\|_{L^{3}(\Omega)}+\|\nabla\nu(\phi^{\varepsilon}_{0}) \|_{L^{6}(\Omega)}\| g\|_{L^{3}(\Omega)}
%\\& \leq  C(\phi^{\varepsilon}_{0})\|g\|_{H^{\frac{1}{2}}(\Omega)}
%\end{align*}
%\end{proof}

\subsection{Proof of  Proposition \ref{pjme1}}
 Let $\mathbf{f}\in L^{2}(0,T;H)$ be defined as before.  Assuming that $(\mathbf{f}_{1},f_{2},\mathbf{a},\mathbf{u}_0, \phi'_{0})\in Y_{T}$,  $g_{0}=\mathrm{div}\mathbf{u}_0$, $\mathbf{w}_0=P_{\sigma}\mathbf{u}_0$, then $\mathbf{v}_{0}=(\phi', g, \mathbf{w})|_{t=0}=(\phi'_{0}, g_{0}, \mathbf{w}_0)\in H_{N}^{2}(\Omega)\times L^{2}_{(0)}(\Omega)\times( H^{1}(\Omega)\cap L_{\sigma}(\Omega))=(D(\mathbf{A}),H)_{\frac{1}{2},2}$, here we have used Lemma 4.6 in \cite{HA1}. We can extend $\mathbf{f}$ by $0$ for $t>T$.  Without loss of generality, we may assume that $\mathbf{a} \equiv 0$ as the beginning of this section. By using  Theorem \ref{hme} and proposition \ref{Proposition4.4} , the following system
\begin{align*}
&\partial_{t}\mathbf{v}+\mathbf{A}\mathbf{v}+\mathbf{B}\mathbf{v}=\mathbf{f}, t>0,
\\& \mathbf{v}(0)=\mathbf{v}_0,
\end{align*}
has a solution $\mathbf{v}$ satisfying
\begin{align*}
&\|\partial_{t}\mathbf{v}, \mathbf{A}\mathbf{v}+\mathbf{B}\mathbf{v}\|_{L^{2}(0,\infty;H)}\leq C(\|\mathbf{f}\|_{L^{2}(0,\infty;H)}+\|\mathbf{v}_0\|_{(D(\mathbf{A}),H)_{\frac{1}{2},2}}).
\end{align*}
By direct computation, for $0<T\leq T_{0}$, we have
\begin{align*}
&\|\mathbf{v}\|_{L^{2}(0,T;H)}\leq C(T_{0})(\|\partial_{t}\mathbf{v}\|_{L^{2}(0,T;H)}+\|\mathbf{v}_0\|_{H}).
\end{align*}
Therefore,
\begin{align}\label{IV}
\|\partial_{t}\mathbf{v},\mathbf{v},\mathbf{A}\mathbf{v}+\mathbf{B}\mathbf{v}\|_{L^{2}(0,T;H)}&\leq C(T_{0})(\|\mathbf{f}\|_{L^{2}(0,T;H)}+\|\mathbf{v}_0\|_{(D(\mathbf{A}),H)_{\frac{1}{2},2}})
\\&\nonumber \leq C(T_{0})(\|\mathbf{f}_{1}\|_{L^{2}(0,T;L^{2}(\Omega))}+\|{f}_{2}\|_{L^{2}(0,T;H^{1}(\Omega))}+\|\mathbf{v}_0\|_{(D(\mathbf{A}),H)_{\frac{1}{2},2}}).
\end{align}
 This implies that $\mathbf{v}=(\phi', g, \mathbf{w})$ solves \eqref{model-linear-phi0-2}-(\ref{model1-102}). It follows from direct computations that $(\mathbf{u},\phi')$ is a solution \eqref{model-linear-phi0-1}-\eqref{model-linear-phi0-6}. That is
\begin{align*}
L(\phi_{0})\left(\begin{array}{c}
\phi' \\
\mathbf{u}
\end{array}\right)=
\left(\begin{array}{c}
\mathbf{f}_{1} \\
f_{2}\\
0
\end{array}\right).
\end{align*}
Our last task is to verify  $(\mathbf{u}, \phi')\in X_{T}$ and the continuity of $L^{-1}(\phi_{0}):Y_{T}\rightarrow X_{T}$. It suffice to employing (\ref{IV}) and Proposition \ref{Proposition4.4} to prove
 \begin{align*}
 \|(\phi',\mathbf{u})\|_{X_{T}}&=\|\phi'\|_{H^{1}(0,T; H^{1}(\Omega))}+\|\mathbf{u}\|_{H^{1}(0,T; L^{2}(\Omega))}+\|\phi'\|_{L^{2}(0,T; H^{3}(\Omega))}+\|\mathbf{u}\|_{L^{2}(0,T; H^{2}(\Omega))}
 \\&\lesssim\|\phi'\|_{H^{1}(0,T; H^{1}(\Omega))}+\|g\|_{H^{1}(0,T; H^{-1}_{(0)}(\Omega))}+\|\mathbf{w}\|_{H^{1}(0,T; L_{\sigma}^{2}(\Omega))}
\\&\quad+\|\phi'\|_{L^{2}(0,T; H^{3}(\Omega))}+\|g\|_{L^{2}(0,T; H^{1}(\Omega))}+\|\mathbf{w}\|_{L^{2}(0,T; H^{2}(\Omega))}
\\&\lesssim\|\partial_{t}\mathbf{v},\mathbf{v},\mathbf{A}\mathbf{v}+\mathbf{B}\mathbf{v}\|_{L^{2}(0,T;H)}
\\&\leq C(T_{0})(\|\mathbf{f}\|_{L^{2}(0,T;H)}+\|\mathbf{v}_0\|_{(D(\mathbf{A}),H)_{\frac{1}{2},2}})
\\&\leq C(T_{0})(\|\mathbf{f}_{1}\|_{L^{2}(0,T;L^{2}(\Omega))}+\|{f}_{2}\|_{L^{2}(0,T;H^{1}(\Omega))}+\|\mathbf{v}_0\|_{(D(\mathbf{A}),H)_{\frac{1}{2},2}})
\\&\leq C(T_{0})\|(\mathbf{f}_{1},{f}_{2},0,\phi'_{0},\mathbf{u}_{0})\|_{X_{T}}.
\end{align*}
  %Furthermore, we can prove the continuity of $(L(\phi_{0}))^{-1}$ by argument at the beginning of this section and following argument.
%\begin{align*}
% \|(\phi',\mathbf{u})\|_{X_{T}}&\leq C(T_{0})(\|\mathbf{f}_{1}\|_{L^{2}(0,T;L^{2}(\Omega))}+\|{f}_{2}\|_{L^{2}(0,T;H^{1}(\Omega))}+\|\mathbf{v}_0\|_{(D(\mathbf{A}),H)_{\frac{1}{2},2}})
% \\& \leq C(T_{0})(\|\mathbf{f}_{1}\|_{L^{2}(0,T;L^{2}(\Omega))}+\|{\tilde{f}}_{2}\|_{L^{2}(0,T;H^{1}(\Omega))}
% \\&\quad+\|\mathrm{div}(\phi_{0}\mathbf{u})\|_{L^{2}(0,T;H^{1}(\Omega))}+\|\mathbf{v}_0\|_{(D(\mathbf{A}),H)_{\frac{1}{2},2}})
% \\& \leq C(T_{0})(\|\mathbf{f}_{1}\|_{L^{2}(0,T;L^{2}(\Omega))}+\|{\tilde{f}}_{2}\|_{L^{2}(0,T;H^{1}(\Omega))}
% \\&\quad+\|\mathrm{div}(\phi_{0}\mathbf{u})\|_{L^{2}(0,T;H^{1}(\Omega))}+\|\mathbf{v}_0\|_{(D(\mathbf{A}),H)_{\frac{1}{2},2}}).
%\end{align*}
Hence, we prove Proposition \ref{pjme1}.
\appendix
\section{Stokes Operator with Navier Boundary Conditions}
\label{sec:StokesNavier}
This appendix is devoted to recalling some results for the Stokes operator with variable viscosity
in the case of Navier boundary conditions. For more information, we refer to \cite{HA11}.

First, we state the Korn inequality by Theorem 3.5 of \cite{JI} as
\begin{align*}
&\|\mathbf{u}\|_{H^{1}(\Omega)}\leq C\|\mathbb{ D}\mathbf{u}\|_{L^{2}(\Omega)},
\end{align*}
for all $\mathbf{u}\in H^{1}_{n}(\Omega)$.

\begin{Lemma}[\cite{HA1}] \label{Appendix A}
Assume that $\phi_{0}\in W^{1}_{q}(\Omega)$, $q>d$, then $A_{a}\mathbf{w}=-P_{\sigma}\mathrm{div}(\tfrac{1}{\rho_{0}}S(\phi_{0},\mathbf{w}))$ where $\mathbf{w}\in D(A_{a})=\{\mathbf{w}\in H^{2}(\Omega)\cap L_{\sigma}^{2}(\Omega): \mathbf{w}\cdot \mathbf{n}|_{\partial\Omega}=0 ,
T_{a}\mathbf{w}=0\}$ is a bounded, self-adjoint and invertible operator.
\end{Lemma}

%\yadong{Why Lemma \ref{Appendix B}? Compare to Lemma \ref{lemma4.5}?}
%\begin{Lemma}\label{Appendix B}\cite{HA1}
%Assume that $\phi_{0}\in W^{2}_{q}(\Omega)$, $q>d$. Then
%\begin{align*}
%&(L_{\sigma}^{2}(\Omega),D(A_{D}))_{\frac{1}{2},2}=H^{1}(\Omega)\cap L_{\sigma}^{2}(\Omega).
%\end{align*}
%\end{Lemma}

\section{Embedding and Multiplication Properties}
\label{sec:Embedding}
We summarize here  a few  basic results about embedding and multiplication properties {on a three dimensional bounded domain with sufficiently smooth boundary.}

\begin{lemma}[\cite{HA}]\label{jme}
Let $X_{0}, X_{1}$ be Banach spaces such that $X_{1}\hookrightarrow X_{0} $ densely, then
\begin{align*}
W_{p}^{1}(0,T; X_{0})\cap L^{p}(0,T; X_{1})\hookrightarrow BUC([0,T]; (X_{0},X_{1})_{1-\frac{1}{p},p}),
  \end{align*}
continuously for $0 < T < \infty$ and $1\leq p<\infty$.
\end{lemma}

\begin{lemma}[\cite{BH,JJ2}]\label{jme1}
Given $g\in B_{p,1}^{\frac{d}{p}}(\Omega)$ and $f\in H^{1}(\Omega)$, it holds
\begin{align*}
&\|fg\|_{H^{1}(\Omega)}\leq C_{p}\|f\|_{H^{1}(\Omega)}\|g\|_{B_{p,1}^{\frac{d}{p}}(\Omega)}, \text{ provided $p\geq2$}.
  \end{align*}

\end{lemma}

\begin{lemma}[\cite{HA1}]\label{jme7}
If $f\in L^{\infty}(\Omega)\cap H^{2}(\Omega)$ and $g\in L^{\infty}(\Omega)\cap H^{2}(\Omega)$,  then
\begin{align*}
&\|fg\|_{H^{2}(\Omega)}\leq C(\|f\|_{L^{\infty}(\Omega)}\|g\|_{H^{2}(\Omega)}+\|f\|_{H^{2}(\Omega)}\|g\|_{L^{\infty}(\Omega)}).
  \end{align*}
\end{lemma}

\begin{lemma}[\cite{HA1}]\label{jme2}
Let $ X_{T}^{2}\coloneqq\{\phi\in H^{1}(0,T;H^{1}(\Omega))\cap  L^{2}(0,T;H^{3}(\Omega)): \mathbf{n}\cdot\nabla\phi|_{\partial\Omega}=0\} $ as in Section \ref{sec:proofofThm}. Then it follows
\begin{align*}
X_{T}^{2}&\hookrightarrow C^{\frac{1}{2}}([0,T];H^{1}(\Omega))\cap BUC([0,T];H^{2}(\Omega))\\&\hookrightarrow C^{\frac{1}{8}}([0,T];H^{\frac{7}{4}}(\Omega))\hookrightarrow C^{\frac{1}{8}}([0,T];B_{3,1}^{1}(\Omega)) \hookrightarrow C^{\frac{1}{8}}([0,T];L^{\infty}(\Omega)).
  \end{align*}
\end{lemma}

\begin{lemma}[\cite{MR}]\label{jme5}
Let
$H^{s,2s}_{p}((0,T)\times\Omega;E)=H^{s}_{p}(0,T;L^{p}(\Omega;E))\cap L^{p}(0,{H_p^{2s}(\Omega;E))}$
with $0<s\leq1$, {$2s \in \mathbb{N}$,} $ 1 < p < \infty $ and $E$ be a Banach space of class $UMD$ which means that
the Hilbert transform is bounded on $E$. Then the trace operator $tr|_{\partial\Omega}$ is continuous as
$$tr|_{\partial\Omega}: H^{s,2s}_{p}((0,T)\times\Omega;E)
\to
H^{s-\frac{1}{2p},2s-\frac{1}{p}}_{p}((0,T)\times\partial\Omega;E).$$ In particularly, $H^{s,2s}_{2}\triangleq H^{s,2s}, H^{s-\frac{1}{4},2s-\frac{1}{2}}_{2}\triangleq H^{s-\frac{1}{4},2s-\frac{1}{2}}$.
\end{lemma}

\begin{lemma}\label{jme6}
If $f\in  BUC(0,T;H^{2}(\Omega))\cap H^{1}(0,T;H^{1}(\Omega))\cap H^{\frac{1}{2}}(0,T;H^{2}(\Omega))$, $g\in L^{\infty}(0,T;L^{2}(\Omega))\cap H^{\frac{1}{2}}(0,T;L^{2}(\Omega))$ and $f|_{t=0}=0$, we have
\begin{align*}
\|fg\|_{H^{\frac{1}{4},\frac{1}{2}}(S(T))}&\leq CT^{\frac{1}{8}}\|f\|_{C^{\frac{1}{8}}(0,T;B^{1}_{3,1}(\partial\Omega))}
\|g\|_{L^{2}(0,T;H^{1}(\Omega))}
\\&\quad+CT^{\frac{1}{6}}\|f\|_{H^{\frac{2}{3}}(0,T;L^{\infty}(\Omega))}
\|g\|_{L^{\infty}(0,T;L^{2}(\Omega))}
\\&\quad+CT^{\frac{1}{8}}\|f\|_{C^{\frac{1}{8}}(0,T;L^{\infty}(\Omega))}
\|g\|_{H^{\frac{1}{2}}(0,T;L^{2}(\Omega))}.
\end{align*}
\end{lemma}
\begin{proof} According to lemma \ref{jme1} ($p=3,d=3$), Lemma \ref{jme2} and Lemma \ref{jme5}($p=2,s=\frac{1}{2}$), one has
\begin{align*}
\|fg\|_{H^{\frac{1}{4},\frac{1}{2}}(S(T))}&\leq C\|fg\|_{H^{\frac{1}{2},1}(Q(T))}
\\&\leq C\|fg\|_{L^{2}(0,T;H^{1}(\Omega))}+C\|fg\|_{\dot{H}^{\frac{1}{2}}(0,T;L^{2}(\Omega))}
\\&\leq C\|f\|_{L^{\infty}(0,T;B^{1}_{3,1}(\Omega))}
\|g\|_{L^{2}(0,T;H^{1}(\Omega))}
\\&\quad+\bigg(\int_0^{T}\int_0^{T}\frac{\|(f(t)-f(\tau))g(t)\|_{L^{2}(\Omega)}^{2}}{|t-\tau|^{2}}\mathrm dt \mathrm d\tau\bigg)^{\frac{1}{2}}
\\&\quad+\bigg(\int_0^{T}\int_0^{T}\frac{\|(g(t)-g(\tau))f(\tau)\|_{L^{2}(\Omega)}^{2}}{|t-\tau|^{2}}\mathrm dt \mathrm d\tau\bigg)^{\frac{1}{2}}
\\&\leq C\|f\|_{L^{\infty}(0,T;B^{1}_{3,1}(\Omega))}
\|g\|_{L^{2}(0,T;H^{1}(\Omega))}\\&\quad+\bigg(\int_0^{T}\int_0^{T}\frac{\|f(t)-f(\tau)\|_{L^{\infty}(\Omega)}^{2}\|g(t)\|_{L^{2}(\Omega)}^{2}}{|t-\tau|^{2}}\mathrm dt \mathrm d\tau\bigg)^{\frac{1}{2}}
\\&\quad+\bigg(\int_0^{T}\int_0^{T}\frac{\|g(t)-g(\tau)\|_{L^{2}(\Omega)}^{2}\|f(\tau)\|_{L^{\infty}(\Omega)}^{2}}{|t-\tau|^{2}}\mathrm dt \mathrm d\tau\bigg)^{\frac{1}{2}}
%\\&\quad+\bigg(\int_0^{T}\int_0^{T}\frac{\|f(t)-f(\tau)\|_{L^{\infty}(\Omega)}^{2}\|g(t)-g(\tau)\|_{L^{2}(\Omega)}^{2}}{|t-\tau|^{2}}\mathrm dt \mathrm d\tau\bigg)^{\frac{1}{2}}
\\&\leq CT^{\frac{1}{8}}\|f\|_{C^{\frac{1}{8}}(0,T;B^{1}_{3,1}(\Omega))}
\|g\|_{L^{2}(0,T;H^{1}(\Omega))}
+CT^{\frac{1}{6}}\|f\|_{H^{\frac{2}{3}}(0,T;L^{\infty}(\Omega))}
\|g\|_{L^{\infty}(0,T;L^{2}(\Omega))}
\\&\quad+CT^{\frac{1}{8}}\|f\|_{C^{\frac{1}{8}}(0,T;L^{\infty}(\Omega))}
\|g\|_{H^{\frac{1}{2}}(0,T;L^{2}(\Omega))}.
\end{align*}
Thus the proof of this lemma is completed.
\end{proof}

\section*{Acknowledgments}
We would also like to
thank the anonymous referee for helpful comments leading to an improvement of this paper.
\subsection*{Funding}
M. Fei is supported partly by NSF of China under Grant No.~12271004, No.~12471222 and NSF of Anhui
Province of China under Grant No.~2308085J10.  D. Han is supported partly by NSF grant DMS-2310340. Y. Liu was partially supported by the startup funding from Nanjing Normal University, the NSF of Jiangsu Province  of China under Grant No.~BK20240572, the NSF of Jiangsu Higher Education Institutions of China under Grant No.~24KJB110020, and the China Postdoctoral Science Foundation under Grant No.~2025M773078.

\section*{Compliance with Ethical Standards}
	\subsection*{Date avability}
	Data sharing not applicable to this article as no datasets were generated during the current study.
	\subsection*{Conflict of interest}
	The authors declare that there are no conflicts of interest.

\end{document}